%% file: BM_QSD_v2.tex
\documentclass[12pt]{article}
\usepackage{hyperref}
\usepackage{amssymb,amsfonts,amsmath,amsthm}
\usepackage{comment}
\usepackage{enumitem}
\usepackage{array}
\usepackage{multirow}
\usepackage{makecell}
\usepackage{tablefootnote}
\usepackage{graphicx}
\usepackage{subcaption}
\usepackage[normalem]{ulem}
\usepackage[colorinlistoftodos,prependcaption,textsize=tiny]{todonotes}

\usepackage{color}
\usepackage{graphicx}

\newcommand{\rr}{{\mathbb R}}

\newcommand{\beq}{\begin{eqnarray*}}
\newcommand{\feq}{\end{eqnarray*}}
\newcommand{\beqn}{\begin{eqnarray}}
\newcommand{\feqn}{\end{eqnarray}}

\newtheorem{theorem}{Theorem}
\makeatletter \@addtoreset{theorem}{section}\makeatother

\newtheorem{lemma}[theorem]{Lemma}
\newtheorem{assume}[theorem]{Assumption}

\newtheorem*{theorem*}{Theorem}

\newtheorem{proposition}[theorem]{Proposition}
\newtheorem{corollary}[theorem]{Corollary}
\newtheorem{definition}[theorem]{Definition}
\newtheorem{meta}[theorem]{Principle}

\newtheorem{ex}[theorem]{Example}
\begin{document}
\title{Quasi-Limiting Behavior of Drifted Brownian Motion}
\author{
SangJoon Lee\thanks{Department of Mathematics, University of Connecticut, Storrs, CT 06269-1009, USA; \newline e-mail:
sangjoon.lee@uconn.edu}
\and
Iddo Ben-Ari\thanks{Department of Mathematics, University of Connecticut, Storrs, CT 06269-1009, USA; \newline e-mail:
iddo.ben-ari@uconn.edu}
}
\maketitle
\begin{abstract}
A Quasi-Stationary Distribution  (QSD) for a Markov process with an almost surely hit absorbing state is a time-invariant initial distribution for the process conditioned on not being absorbed by any given time. An initial distribution for the process is in the domain of attraction of some QSD $\nu$  if the distribution of the process a time $t$, conditioned not to be absorbed by time $t$ converges to $\nu$ as $t$ tends to infinity.  We study Brownian motion with constant drift on the half line $[0,\infty)$ absorbed at $0$. Previous work by Martinez {\it et al.} \cite{QSD_1994} \cite{QSD_1998} identifies all QSDs and provides a nearly complete characterization for their domains of attraction. Specifically, it was shown that if the distribution a well-defined exponential tail (including the case of lighter than any exponential tail), then it is in the domain of attraction of a QSD determined by the exponent. In this work we expand the discussion regarding the dependence on the initial distribution through 
\begin{enumerate}
    \item Obtaining a new approach to existing results, explaining  the direct relation between a QSD and an initial distribution in its domain of attraction; and 
    \item Considering  a wide class of heavy-tailed initial distributions, where non-trivial limits are obtained under appropriate scaling. 
\end{enumerate}
\end{abstract}

\input{sec1.tex}

\input{sec2.tex}

\input{sec3.tex}

\input{sec4.tex}

\input{sec5.tex}

\input{appendix.tex}

\section*{Acknowledgement} The authors would like to express their deep gratitude to an anonymous referee whose input was invaluable to the presentation of our results and helped correcting errors and omissions.  
\bibliographystyle{amsplain}
\bibliography{BM_QSD_bibliography}

\end{document}

%% file: sec1.tex
\section{Introduction}

Here we review the origin and some well-known results of the study of QSDs. In section \ref{sec:gen}, we will present the general definition of QSD and related theorems. In section \ref{sec:QSD}, we will  introduce the specific model we work in this paper, and present some previous results on the model.

\subsection{Definitions and General Results}\label{sec:gen} 
Consider ${\bf X} =(X_t:t\ge 0)$, a Markov process on  $\mathbb{R}_+=[0,\infty)$ with $0$ as a unique absorbing state. Let 
$$ \tau = \inf\{t\ge 0:X_t = 0\}.$$ 
We will work under the assumption 
\begin{equation}
\label{eq:zero_attainable}
P_x (\tau < \infty) = 1,\mbox{ for all }x \in \rr_+. 
\end{equation} 

The notation $P_x$ is a shorthand for  the distribution of ${\bf X}$ with initial distribution, the distribution of $X$, equal to the Dirac-delta measure at $x$.


If $\pi$ is a stationary distribution for ${\bf X}$, then \eqref{eq:zero_attainable} guarantees that $\pi = \delta_0$, \cite[Section 2.2]{QSD_book}.
While this result is not very interesting, the  distribution of the process and particularly of $X_t$ {\it conditioned } on $\{\tau >t\}$, is in general far from trivial. 
This naturally leads to the following ``conditional'' analog for a stationary distribution: 
\begin{definition} 
\label{def:qsd} 
The probability distribution $\pi$ is a Quasi-Stationary Distribution (QSD) for ${\bf X}$ if 
$$P_\pi(X_t \in \cdot\; |  \; \tau > t) = \pi \text{ for all } t >0.$$
\end{definition}
A seemingly more relaxed definition, in the spirit of ergodic theorems for Markov Chains,  is the following: 
\begin{definition}
\label{def:qld} 
A probability distribution $\pi$ is a Quasi-Limiting Distribution (QLD) for ${\bf X}$ if for some  $\mu$, 
\begin{equation} 
\label{eq:QLD} 
\lim_{t \rightarrow \infty} P_\mu(X_t \in  \cdot \; | \; \tau > t)=\pi,
\mbox{ in distribution}. 
\end{equation} 
\end{definition}
where, as usual, $P_\pi$ and $P_\mu$ are shorthand for the distribution of ${\bf X}$ with initial distribution  equal to  $\pi$ or $\mu$, respectively.  QLDs corresponding to an initial distributions which are  Dirac-delta (or, more generally, compactly supported initial distributions) are known as Yaglom limits. \\

Of course, a QSD is a QLD. A partial converse holds under a standard regularity assumption (the Feller property):
\begin{proposition}
\label{prop:QLD} 
Suppose that for every $t>0$ and continuous and bounded function $f$ on $(0,\infty)$, the function $x\to E_x [ f(X_t),\tau>t]$ is continuous. Then every QLD for ${\bf X}$ is a QSD for ${\bf X}$. 
\end{proposition}
For the sake of completeness, we provide a proof in Appendix \ref{appendix1}.\\

We comment that QLDs are very often defined by requiring a {\it pointwise limit} rather than limit in distribution. That is \eqref{eq:QLD} in Definition \ref{def:qld} is replaced by 
\begin{equation} \label{eq:QLD_pointwise} \lim_{t\to\infty} P_{\mu} (X_t \in A| \tau>t)=\pi(A)\mbox{  for all measurable }A\subseteq (0,\infty).\end{equation} 
With this definition the conclusion of Proposition \ref{prop:QLD} holds without the additional regularity condition we imposed.  See \cite[Definition 1 and Proposition 1]{villemonais2011}.\\

As in the sequel we will only work with processes satisfying the condition in the proposition, we always consider QLDs as QSDs. In light of the above, when  $\mu$ and $\pi$ are as in Definition \ref{def:qld}, we say that $\mu$ is in the {\it domain of attraction of } the QSD $\pi$. Of course, the domain of attraction of any QSD contains itself.  \\ 

Figure \ref{fig:qsd} illustrates the difference between the unconditioned process, the process that is required to be positive only at the given time, and the process that is required to never hit $0$ up to the given time.\\

\begin{figure}
\centering
\begin{subfigure}{0.4\textwidth}
\includegraphics[width=0.9\linewidth]{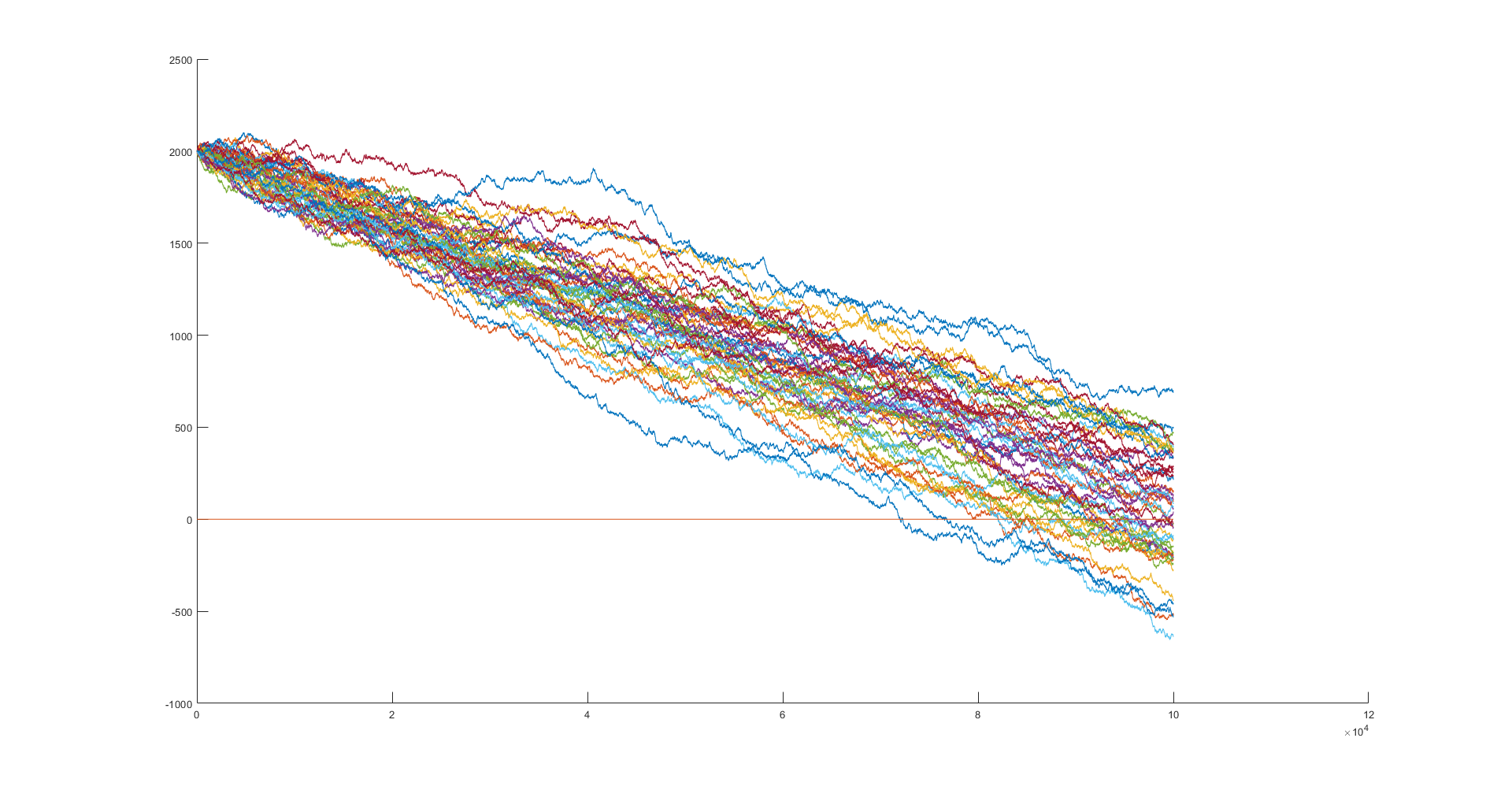}
\caption{Sample paths of 1-dimensional Brownian Motion with constant negative drift $-0.02$, with fixed initial state $X_0 = 2000$}
\end{subfigure}
\quad
\begin{subfigure}{0.4\textwidth}
\includegraphics[width=0.9\linewidth]{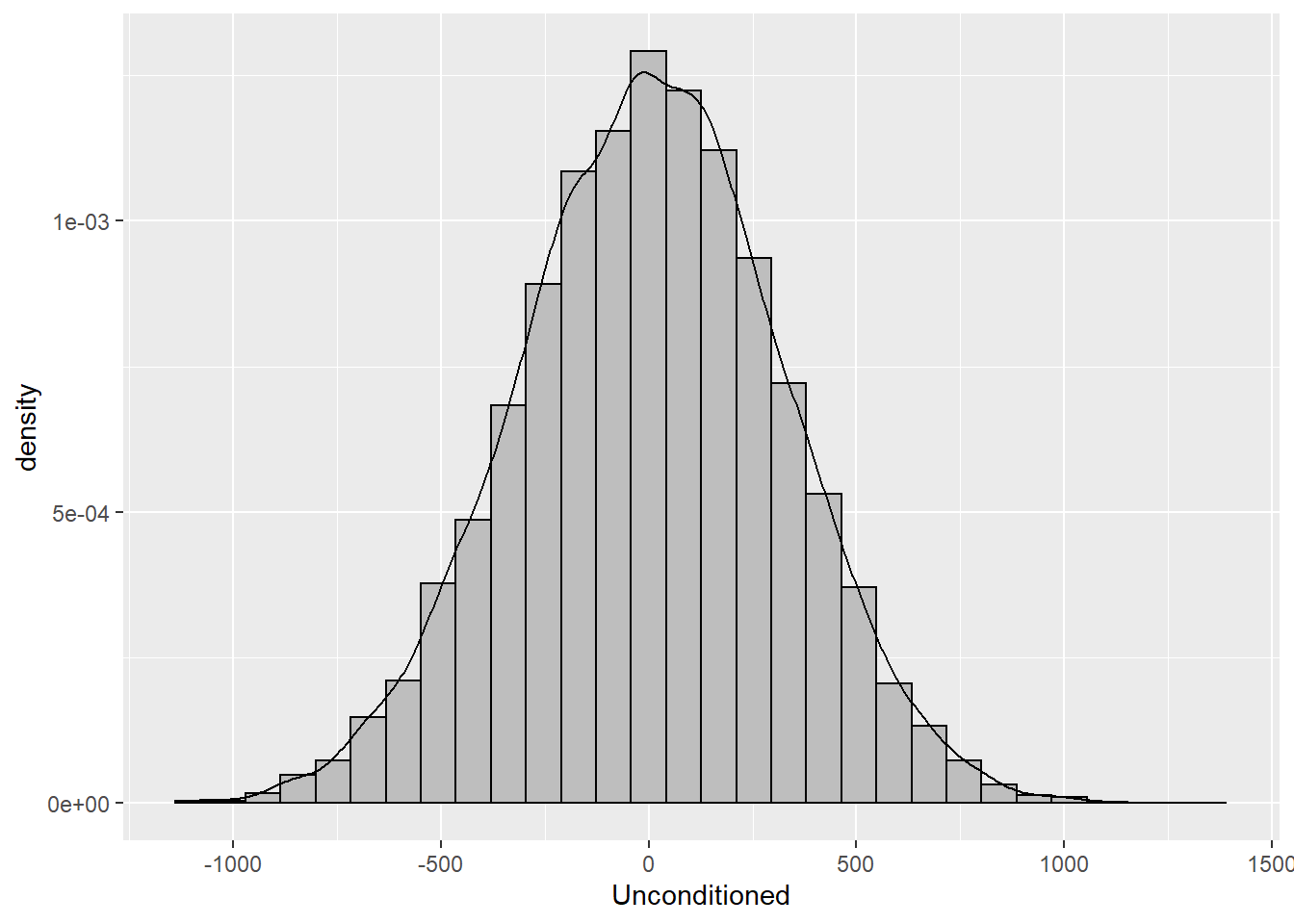}
\caption{PDF plot of 1-dimensional Brownian Motion with same drift and initial state, at $t = 100000$. Sample size is $10000$. As expected, $X_{10000}$ follows a Gaussian distribution.}
\end{subfigure}
\label{fig:non_qsd}
\vskip\baselineskip
\begin{subfigure}{0.4\textwidth}
\includegraphics[width=0.9\linewidth]{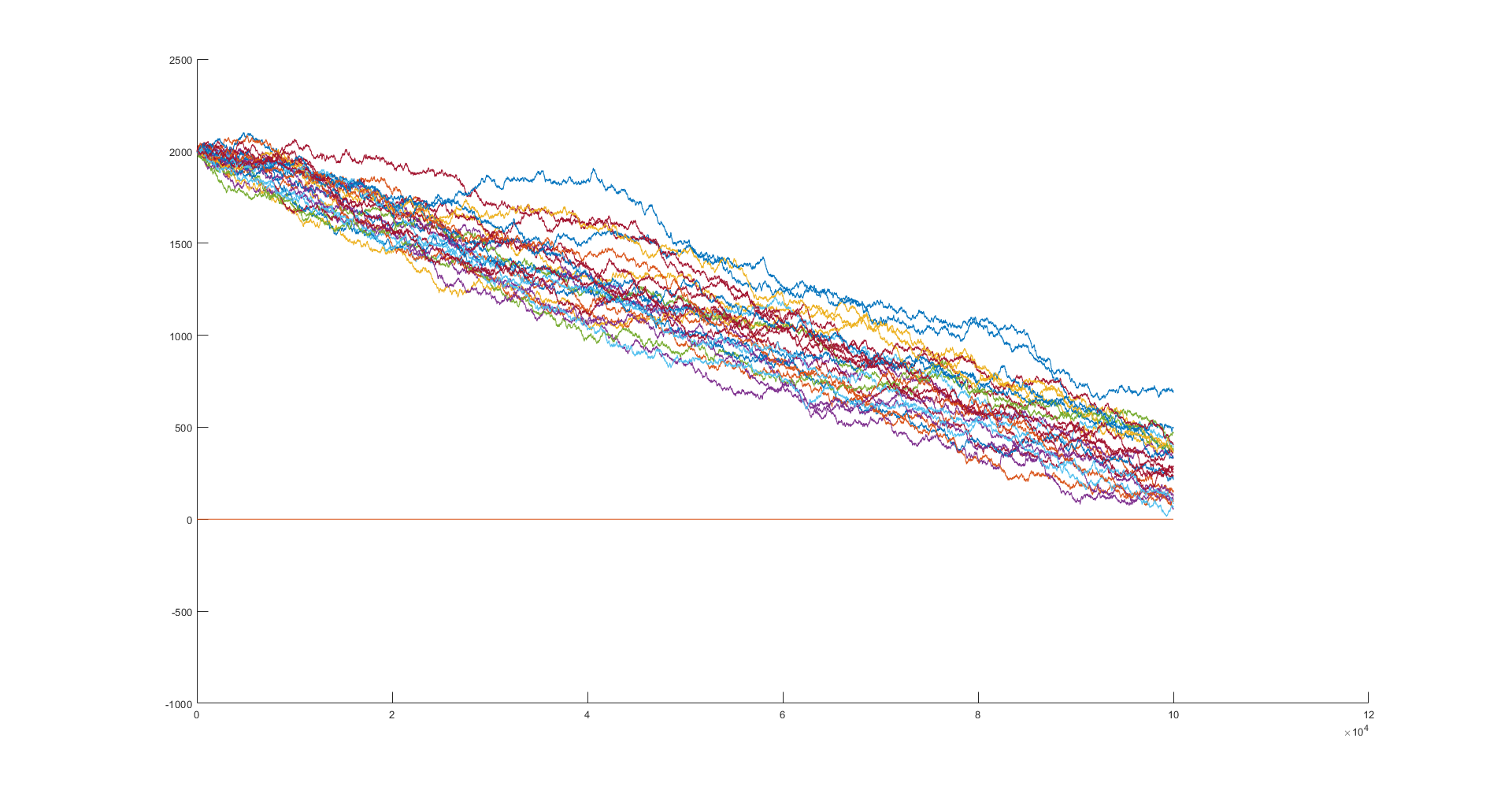}
\caption{Sample paths of same processes, conditioned not to be absorbed by $t = 100000$}
\end{subfigure}
\quad
\begin{subfigure}{0.4\textwidth}
\includegraphics[width=0.9\linewidth]{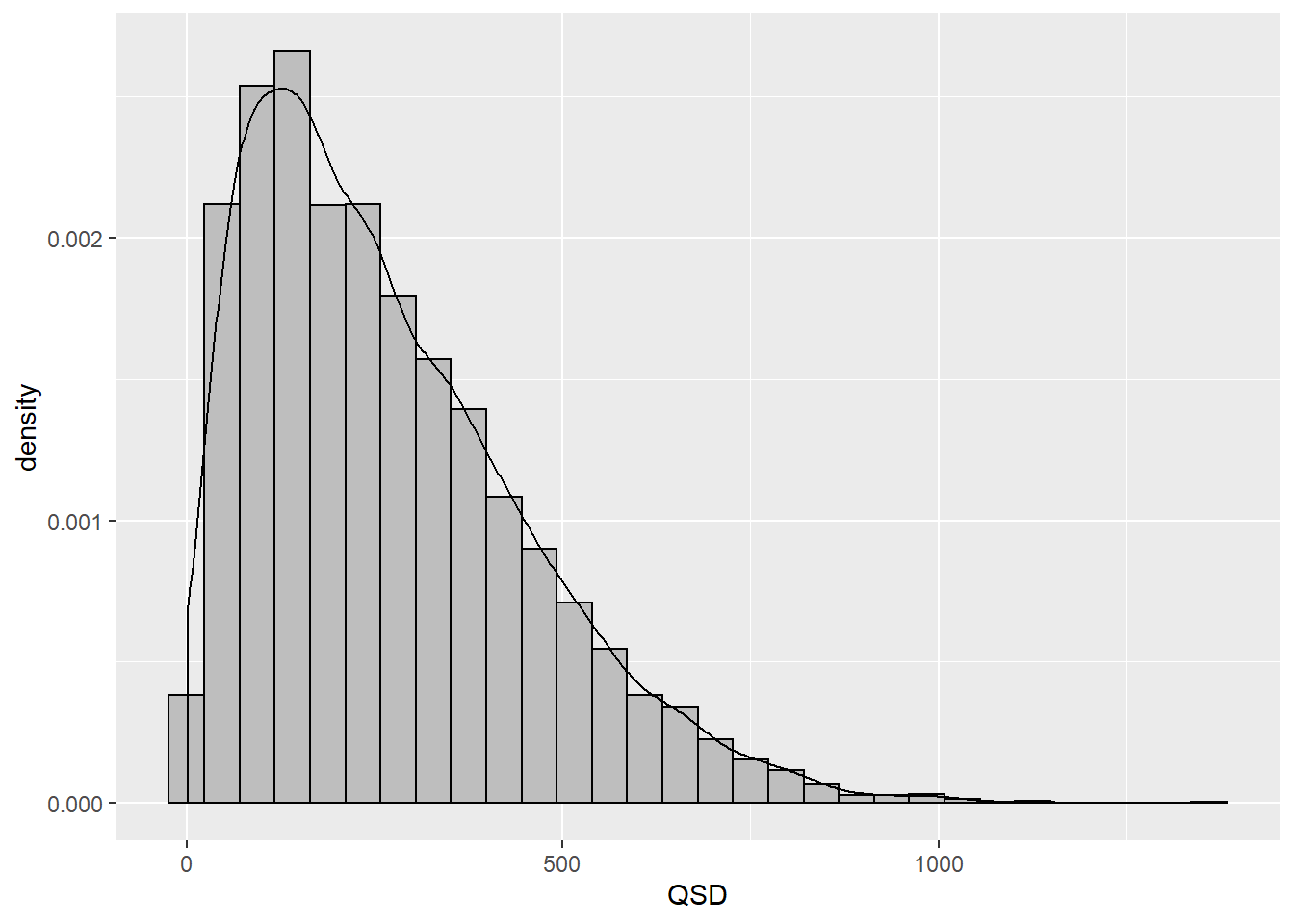}
\caption{PDF plot of the same sample processes and same condition. Unlike above, this distribution has exponential tail. Also, the density near $0$ drops significantly in this setting.}
\end{subfigure}
\caption{Illustration between unconditioned stochastic process and process conditioned not to be absorbed by a given time}
\label{fig:qsd}
\end{figure}

Unlike uniqueness of stationary distribution under irreducibility assumptions, QSDs are in general not unique, and typically a continuum of QSDs exists. Notable exceptions of this are Markov chains on finite state spaces (with a unique absorbing state) and certain diffusion processes on bounded domains absorbed at the boundary. One strategy of finding QSDs is  to study the quasi-limiting behavior under different initial distributions. When the class of QSDs is known it is natural to ask what is the domain of attraction of each.  \\



The concept of QSD is fairly intuitive and straightforward, as the idea was first introduced as early as 1931 by Wright \cite{wright1931}, and the terms related to QSD have been crystallized in 1950s by Bartlett \cite{bartlett1957} \cite{bartlett1960}. Mathematically, Yaglom \cite{yaglom1947} first showed an explicit solution to a limiting conditional distribution for the for the subcritical Bienaym\'e-Galton-Watson
branching process. In the discrete setting there are detailed results for some specific models; for example, explicit description of QSDs are known for certain birth-and-death processes \cite[Theorem 5.4]{QSD_book}. As for uniqueness, a necessary and sufficient conditions for birth-and-death processes were obtained by van Doorn \cite{van1991quasi} Mart\'{i}nez, San Mart\'{i}n and Villemonais later generalized the result to countable state processes \cite{QSD_2014}. For other discrete state space models, Buiculescu studied QSDs for multi-type Galton-Watson processes \cite{buiculescu1975}, and Ferrari and Mari\'{c} discussed QSDs approximated by Fleming-Viot processes \cite{ferrari2007}. A survey of results provided by van Doorn and Pollett \cite{doorn2013} gives a comprehensive view of the progress on the discrete space models.\\

Our work is on Brownian Motion with constant drift, which is one among a few models where a lot is known explicitly, in part because it is a Gaussian process. Our main object of interest is the dependence on the initial distribution, and is in continuation to the works of  Mart\'{i}nez and San Mart\'{i}n who identified all QSDs for the model \cite{QSD_1994}, and later identified the domain of attraction for each QSD \cite{QSD_1998}. Rates of convergence to Yaglom limits are also studied by Polak and Rolski \cite{polak_rolski}, and O\c{c}afrain \cite{ocafrain}. Another diffusion - notably also a Gaussian processes -  where explicit results are known is the  Ornstein-Uhlenbeck process:  Lladser and San Mart\'{i}n \cite{lladser2000} classified QSDs through their domains of attraction.  Ye \cite{jun2008} identified the Yaglom limit for fractional-dimensional radial Ornstein-Uhlenbeck processes. As for general theory for diffusion processes, there has been much work and progress on conditions for existence and uniqueness of QSDs and on convergence to the Yaglom limit.  Here is a partial list of references:   Pinsky \cite{pinsky} (smooth bounded domains with absorption on the boundary), Cattiaux, Collet, Lambert, Amaury, Mart\'{\i}nez, M\'{e}l\'{e}ard  and San Mart\'{\i}n \cite{cat_collet_et_al},  Steinsaltz and Evans \cite{steinsaltz2007}, Kolb and Steinsaltz \cite{kolb2012}, and Hening and Kolb \cite{hening2019} (uniqueness and convergence for one dimensional diffusions) and Champangnat and Villemonais \cite{champ_vil} (rates of convergence for one-dimensional diffusions). \\

We close this section with the well-known properties related to QSDs and QLDs.  
\begin{theorem}
\cite[Theorem 2.2]{QSD_book}
\label{th:exp_tau}
Suppose that $\pi$ is a QSD. Then under $P_\pi$,  $\tau$ is exponentially distributed with parameter $\lambda_\pi>0$. 
\end{theorem}
\begin{proposition}\label{prop:heavytail}
Let the assumption of Proposition \ref{prop:QLD} hold.  Let $\mu$ is in the domain of attraction of the QSD $\pi$. Then for every $\epsilon>0$, 
$$P_\mu(\tau > t) = o(e^{-(\lambda_\pi-\epsilon)t}),$$ 
where $\lambda_\pi$ is as in Theorem \ref{th:exp_tau}.
\end{proposition}
We give an elementary proof in Appendix \ref{appendix2}. See also \cite[Proposition 5]{villemonais2011} for a sharper result under slightly stronger assumptions. 

\subsection{Quasi Stationarity for Drifted BM}\label{sec:QSD}
In this section and the sequel we will work under the following: 
\begin{assume}
\label{assume:drifted_BM}
${\bf X}$ is Brownian Motion (BM) with constant negative drift $-\alpha$, $\alpha>0$, on $\rr_+$ absorbed at $0$. 
\end{assume}

Analytically, BM with constant drift $-\alpha$ on  $\rr_+$  absorbed at $0$ is the sub-Markovian process generated by  ${\cal L}_{\alpha}$, which for each $u$ satisfying $u \in C^2(\rr_+)$ and $u(0) = 0$, 
$$ {\cal L}_\alpha u = \frac 12 u'' - \alpha u'.$$ 

The works by  Mart\'{i}nez, Picco and San Mart\'{i}n \cite{QSD_1994}\cite{QSD_1998} studied QSDs for this class of models. The formal derivation for densities of the QSDs, as presented in their main results, will be given in Appendix \ref{appendix3}.

\begin{theorem} \cite[Proposition 1]{QSD_1994}
Every QSD for ${\bf X}$ is of the form $\pi_\gamma$ for some $\gamma \in [0,\alpha)$. 
\end{theorem} 
\begin{theorem}\cite[Theorem 1.3]{QSD_1998}
\label{th:light} 
The probability measure $\mu$ is in the domain of attraction of $\pi_0$ if 
 $$\liminf_{x\to\infty}  \frac{\ln \mu([x,\infty))}{x}\le - \alpha.$$ 
\end{theorem}
\begin{theorem} \cite[Theorem 1.1]{QSD_1998}
\label{th:medium} 
Let $\rho \in (0,\alpha)$. The probability measure $\mu$ is in the domain of attraction of $\pi_{\alpha - \rho}$ if   $$\lim_{x\to\infty} \frac{\ln \mu([x,\infty))}{x}= - \rho.$$
\end{theorem} 
We note the following: 
\begin{enumerate} 
\item Theorem \ref{th:medium} was proved under the assumption that $\mu$ has a smooth density. 
\item  The limit condition in Theorem \ref{th:medium} is not merely technical.
The authors constructed an example \cite[Theorem 1.4]{QSD_1998} with initial distribution with tail which alternates between two exponential decay rates and which  is not in the domain of attraction of any QSDs. We comment that the method we develop in this paper can provide a simpler construction of such initial distribution.
\end{enumerate} 
\subsection{Organization}
We present our main results in Section \ref{sec:results}, split according to the tail of the initial distribution, considering  initial distributions in the domain of attraction on QSDs in Section \ref{sec:QSD_domain} and heavy tails in Section \ref{sec:heavy}. Our proofs are given split across three sections:  In Section \ref{sec:master_eq}, we present some general tools we will use. In Section \ref{light_tail_proofs} we prove the results from Section \ref{sec:QSD_domain}. In Section \ref{heavy_tail_proofs} we prove the results from Section \ref{sec:heavy}, along with some concrete examples in Section \ref{sec:examples}.

%% file: sec2.tex
\section{Main Results} 
\label{sec:results} 
In this section we will state our main results, by first proposing a principle which we think can help the readers to envision the general classification of quasi-limiting behavior, and then provide the theorems based on the principle. We recall that we are working under Assumption \ref{assume:drifted_BM}.
\newtheorem*{assume:drifted_BM}{Assumption \ref{assume:drifted_BM}}

Our goals are twofold:
\begin{enumerate}
\item Develop a method that would yield alternate proof to Theorems \ref{th:light} and \ref{th:medium}, which can be generalized to other models, as well as leading to complete characterization of the domain of attraction of every QSD. Our results are presented in Section \ref{sec:QSD_domain}. 
\item Characterize the asymptotic behavior when the initial distribution has tails which are heavier than exponential. It is not hard to show,  see Lemma \ref{lem:heavy_tail}, that this class of initial distributions is not in the domain of attraction of any QSD. Our results are presented in Section  \ref{sec:heavy}. 
\end{enumerate}

\subsection{Domain of Attraction of QSDs} 
\label{sec:QSD_domain} 
As at its core,  the  concept of quasi-stationarity concerns conditional probabilities under events with diminishing probabilities, namely the events $\{\tau>t\}$.  It is therefore natural  to study the rate at their probabilities, $P_\mu (\tau>t)$, tend to zero.  One of the nice properties  of our model is that through Girsanov theorem and the reflection principle (or formulas for Brownian bridges) a closed form formula for these probabilities is readily available. We have: 
\begin{proposition}\label{prop:masterformula}
\begin{equation}
\label{eq:master_formula} 
P_{\mu}( X_t \in dy, \tau > t) =\frac{1} {\sqrt{2\pi t}} \int \exp\left(\alpha x - \frac{\alpha^2 t}{2} - \alpha y\right) \left(e^{-\frac{(x-y)^2}{2t}}-e^{-\frac{(x+y)^2}{2t}}\right)d\mu(x).
\end{equation}
\end{proposition}
Our approach to the problem is to obtain for each initial distribution $\mu$ a family of probability measures $(\nu_t:t\ge 0)$, such that 

\begin{meta}\label{meta:light}
\begin{equation}
    \label{eq:QSD_prinicple} 
    \boxed{ \lim_{t\to\infty} \nu_t = \delta_\gamma} \;\Longrightarrow\; \boxed{\lim_{t\to \infty} P_{\mu} (X_t \in \cdot \;|\; \tau>t)=\pi_\gamma}
\end{equation}
\end{meta}

 The measure $\nu_t$ is defined through its cumulative distribution function $F_{\nu_t}$: 

\begin{equation}\label{eq:nuconst}F_{\nu_t}(z) = C_t \int_{[0,zt]} e^{-x^2/(2t)} e^{\alpha x} d \mu(x)\end{equation}

where $C_t$ is the normalization constant. Table \ref{rho_table} is the summary of our result; it shows the relation between $\mu$, $\nu_t$ and the QLD of $\mu$. \\

\begin{table}[ht] \label{exp_table}
\centering
{\renewcommand{\arraystretch}{2}
\begin{tabular}{ |c|c|c|c|  } 
\hline
$\rho$ & $\lim \nu_t$ & QLD (= QSD) & Example distributions \\
\hline \hline
$\rho \ge \alpha$ & $\delta_0$ & \makecell{$\pi_0$ \\ (Theorem \ref{thm:light_tails})} & \makecell{\\Half-normal distribution \\ Delta distribution \\ \quad} \\ 
\hline
$\alpha > \rho > 0$ & $\delta_{\alpha - \rho}$ & \makecell{$\pi_{\alpha - \rho}$ \\ (Theorem \ref{thm:medium})} & \makecell{\\Exponential distribution \\ with rate $\lambda < \alpha$ \\ \quad} \\ 
\hline
$\rho = 0$ & $\delta_\alpha$ & \makecell{\\QLD does not exist: \\ scaling is necessary. \\ See Section \ref{sec:heavy} \\ and Table \ref{beta_table} \\ \quad} & \makecell{\\Pareto distribution \\ Half-Cauchy distribution \\ \quad} \\ 
\hline
\end{tabular}}
\caption{Domain of attraction classified by parameter $\displaystyle \rho = \lim_{x \to \infty} -\frac{\ln \mu([x,\infty))}{x}$}
\label{rho_table}
\end{table}

The key idea in the method is to ``decouple'' the initial distribution from the asymptotic distribution, then identifying the relevant QSD as a member of a one-parameter family selected according to the value of $\gamma$. Indeed, in our model, observe that  the mapping $\gamma \to \pi_\gamma, ~\gamma \in  [0,\alpha)$ as given in \eqref{eq:pigamma} is an explicit  function, with the case $\gamma=0$ is merely a removable singularity and is defined as   $\lim_{\gamma  \to 0+} \pi_\gamma$.\\

We believe that this method has a number of advantages:
\begin{enumerate} 
\item It is more intuitive, simpler and elementary than the previous approach. It lets us understand how the initial distribution actually evolves over time, and at a specific time, which part of the initial distribution have evolved to consist the absolute majority of the process not absorbed.
\item The method allows for expanded characterization of the domain of attraction of QSDs.
\item Our approach simplifies the analysis for the case of a distribution with alternating exponential tails, given in \cite{QSD_1998}, and opens the possibility of studying  general compound-tail distributions.
\end{enumerate} 
We expect this method to be applicable to other models and we hope it can be adopted as a general framework for classifying domain of attraction of QSDs. \\

Our Principle \ref{meta:light} will be employed in two ways. We first observe that
\begin{equation}
\lim_{t \to \infty} \nu_t = \begin{cases} \delta_0 &\Longleftrightarrow \limsup_{x \to \infty} -\frac{\ln \mu([x,\infty))}{x} \ge \alpha \\ \delta_{\alpha - \rho} &\Longleftrightarrow \lim_{x \to \infty} -\frac{\ln \mu([x, \infty))}{x} = \rho < \alpha \end{cases}
\end{equation}

Through application of the approach outlined above we obtain the following results: 

\begin{theorem}
\label{thm:light_tails}
Suppose $\mu$ satisfies the following assumption.
\begin{equation}
    \label{eq:crit_or_lighter}
\rho := \liminf_{x\to\infty} -\frac{\ln  \mu([x,\infty))}{x}\ge \alpha.
\end{equation}
Then 
$$P_\mu (X_t \in \cdot | \tau> t) \to \pi_0.$$ 
\end{theorem}

\begin{theorem}\label{thm:medium}
Suppose $\mu$ satisfies the following assumption,
\begin{equation}\label{assume:medium}
\rho := \lim_{x\to \infty} -\frac{ \ln \mu ([x , \infty)) }{x}  \in (0, \alpha)
\end{equation}
and let the sequence of measure $(\nu_t : t \ge 0)$ defined as \eqref{eq:nuconst}. Then
\begin{equation}\lim_{t \to \infty} \nu_t = \delta_{\alpha - \rho}\end{equation}
and moreover,
\begin{equation}
\lim_{t \to \infty} P_\mu (X_t \in \cdot \; | \; \tau > t) = \pi_{\alpha - \rho}
\end{equation}
\end{theorem}
We will refer to $\mu$ satisfying \eqref{eq:crit_or_lighter} as possessing ``Critical and Super-critical" tails (with critical being an equality), and will prove Theorem \ref{thm:light_tails} in Section \ref{sec:light}.  We will refer to $\mu$ that satisfies \eqref{assume:medium} as possessing  ``Sub-critical Exponential" tails and will prove Theorem \ref{thm:medium} in Section \ref{sec:medium}. 

\subsection{Heavy Tails} 
\label{sec:heavy}
A natural question to ask from \cite{QSD_1998} would be the following:  what happens if the initial distribution is too heavy to be in the domain of attraction of any QSDs?  A first step in this direction is to look for such initial distributions.  In light of Theorems \ref{th:light} and \ref{th:medium}, the following is not surprising: 
\begin{lemma}\label{lem:heavy_tail}
Suppose  $$\lim_{x \to \infty} \frac{\ln  \mu([x, \infty))}{x} = 0.$$ 
Then  $P_{\mu} (\tau>t) $ does not decay exponentially. As a consequence $(P_{\mu} (X_t \in \cdot \; | \; \tau> t):t\ge 0)$ is not tight.  
\end{lemma}
Thus, in order to obtain a non-trivial limit, one has to scale $X_t$ as $t\to\infty$. As we will see, the scaling itself depends on $\mu$. We comment that all of the cases covered in this section correspond to $\nu_t \to \delta_\alpha$ in \eqref{eq:QSD_prinicple}. \\

The next step is to study long-time behavior under such heavier-tailed distributions, and this is the main topic of this part of the project. In order to do so, we mainly rely on the theory of  regularly varying functions \cite{rv_bingham}.

\begin{assume}
\label{assume:nice_heavy_tail}
Suppose $\mu$ is a probability measure satisfying the following: 
\begin{enumerate}
    \item $\mu([x,\infty)) = e^{-F(x)}$, with $F$ smoothly varying \cite[Section 1.8]{rv_bingham} with index parameter $\beta<1/2$. 
    \item There exists a positive function $R(x,c)$ on $\rr_+\times \rr_+$  increasing in $c$, such that for all $c>0$
\begin{equation}
\label{eq:nice_difference} \lim_{x \to \infty} F(x+R(x,c)) - F(x) = c.
\end{equation}
\end{enumerate}
\end{assume} 

Some comments are in order: 
\begin{enumerate} 
\item  Probability measures with regularly varying tails falls into the category $\beta = 0$. Some distinguished cases are the Weibull distribution with $0 < k < 1$, which has a uniform decay rate with $\beta = k$, and the Pareto and Cauchy distributions, both having uniform decay rate with $\beta = 0$. 
\item If $F$ is smooth enough, then 
\begin{equation}\label{taylorcondition}R(x,c) = \frac{c}{F'(x)}\end{equation}
So when $\beta \not= 0$, $R(x,c)$ is a regular varying function with index $\varphi = 1-\beta$.
\item When $\beta=0$ it is more  natural to replace the identity function on the right-hand side of \eqref{eq:nice_difference} with a strictly increasing continuous and nonnegative function $H$ satisfying  $H(0)=0$. 
\end{enumerate}

The main principle we developed to obtain results under Assumption \ref{assume:nice_heavy_tail} is the following. 

\begin{meta}\label{meta:heavy_tail}
\begin{equation}
 \boxed{ \mbox{Assumption \ref{assume:nice_heavy_tail}}} \quad \Rightarrow \quad   
\boxed{ \lim_{t \to \infty} P_\mu\left(X_t > R(t,c) \left. \right| \tau > t\right) = e^{-c}}
\end{equation}
\end{meta}

We note that the assumption $\beta <1$ is vital for this to work, as otherwise the conclusion contradicts the results of previous sections. This is due to the fact that $\beta=1$ is the critical border where the relation between the survival rate $P_\mu(\tau > t)$ and the initial distribution $\mu$ changes. Also, although Lemma \ref{lem:heavy_tail} applies whenever $0 \le \beta < 1$, Principle \ref{meta:heavy_tail} only applies to $0 < \beta < 1/2$. The remaining half $1/2 \le \beta < 1$ is left as an open problem; there is a difficulty in estimating the distribution of the surviving processes in these cases. (See Proposition \ref{subexprop} and \eqref{betacond})
The following theorem is the key result from the above principle. 
\begin{theorem}\label{thm:subexp}
Suppose $\mu([x,\infty)) = \exp(-F(x))$ where $F(x)$ is strictly increasing smoothly varying function with index $\beta < 0.5$. Then 
\begin{equation}
\lim_{t \to \infty} P_\mu\left(\left.X_t > \frac{c}{F'(\alpha t)}  \;\right|\; \tau > t\right) = e^{-c}
\end{equation}
\end{theorem}
Table \ref{beta_table} summarizes our results by showing how $\beta$ relates to some of the well-known distributions, and how they lead to quasi-limiting behavior of such initial distribution. The table also lists a number of concrete cases, all presented in  Section \ref{sec:examples}. 

Throughout the rest of the paper, we will be using some asymptotic notations; $f(t) \sim g(t)$ if $\displaystyle \lim_{t \to \infty} \frac{f(t)}{g(t)} \in (0, \infty)$, and $f(t) \ll g(t)$ if $\displaystyle \lim_{t \to \infty} \frac{f(t)}{g(t)} = 0$.

\begin{table}[ht] 
\begin{minipage}{\textwidth}
\centering
{\renewcommand{\arraystretch}{2}
\renewcommand{\thempfootnote}{\arabic{mpfootnote}}
\begin{tabular}{ |c|c|c|  } 
\hline
$\beta$ & Related result & Example distributions \\
\hline \hline
$\beta > 1$ & Theorem \ref{thm:light_tails} & \makecell{\\Half-normal distribution \\ Compactly supported distributions\\Weibull distribution with \\ shape parameter $k > 1$ \\ \quad} \\
\hline
$\beta = 1$ & \makecell{$\rho =  \lim_{x \to \infty} -\frac{\ln \mu([x,\infty))}{x}$\\ Theorem \ref{thm:light_tails} if $\rho  \ge \alpha$ \\ Theorem \ref{thm:medium} if $\rho < \alpha$}  & \makecell{\\Exponential distribution \\ Erlang distribution\\ \quad} \\
\hline
$\frac{1}{2} \le \beta < 1$ & Unknown  & \makecell{\\Weibull distribution with \\ shape parameter $\frac{1}{2} \le k < 1$ \\ \quad} \\ 
\hline
$0 < \beta < \frac{1}{2}$ & Theorem \ref{thm:subexp}  &\makecell{\\Weibull distribution with \\ shape parameter $k < \frac{1}{2}$, Example \ref{xmpl:weibull} \\ \quad}\\ 
\hline
$\beta = 0$ & \makecell{ $\mu ([x,\infty))$ is RV with index $-\kappa$:\\ Corollary \ref{thm:rv} if $\kappa \ne 0$ \\ Corollary \ref{thm:sv} if $\kappa = 0$} & \makecell{\\Pareto distribution \\ Half-Cauchy distribution, Example \ref{xmpl:half-cauchy}\\ Log-Cauchy distribution\\ \quad } \\
\hline
\end{tabular}}
\caption{Distributions classified by index parameter $\beta$ of $F(x) = -\ln \mu([x,\infty))$}
\label{beta_table}
\end{minipage}
\end{table}

%% file: sec3.tex
\section{Proofs: Base Formula}\label{sec:master_eq}
In this section, we prove Proposition \ref{prop:masterformula}, which is the master formula we use throughout this paper. We will also further explain the intuition behind the sequence of new measure $\nu_t$. Finally, we will introduce the variations of Scheffe's lemma \cite{scheffe}, which is one of the tool for Chapter \ref{light_tail_proofs}. \\

\subsection{Conditional transition density}
When $X_t$ is a drifted Brownian Motion with negative drift $\alpha$, (such that $X_t + \alpha t$ is a standard BM $B_t$)
\begin{equation}
\begin{split}
P_x (X_t \in dy) 
&= \exp\left(\alpha x - \frac{\alpha^2 t}{2} + \alpha y\right) P_x(B_t \in dy)
\end{split}
\end{equation}

We also want to enforce the condition $\tau > t$, where $\tau$ is the hitting time at 0. We can apply the reflection principle to compute $P_x (X_t \in dv, \tau > t)$.
\begin{equation}
\begin{split}
P_x (X_t \in dy, \tau > t) &= \exp\left(\alpha x - \frac{\alpha^2 t}{2} + \alpha y\right) P_x(B_t \in dy, \tau > t) \\
&= \underset{= f(t, x,y)}{\underbrace{\exp\left(\alpha x - \frac{\alpha^2 t}{2} + \alpha y\right) \frac{1}{\sqrt{2 \pi t}}\left( e^{-\frac{(x-y)^2}{2t}} - e^{-\frac{(x+y)^2}{2t}}\right)}}
\end{split}
\end{equation}

Integrating $f(t, x,y)$ with respect to $\mu$ gives \eqref{eq:master_formula}. Furthermore, we can get the survival rate from the above formula as well.

\begin{equation}\label{eq:survival_rate}
P_\mu(\tau > t) = \int_0^\infty \int_0^\infty \mu(x) f(t, x,y) dy dx
\end{equation}

We wrap this section with the principle behind finding the family of probability measures $(\nu_t : t \ge 0)$ in \eqref{meta:light}. From \eqref{eq:survival_rate},
\begin{equation}\label{eq:m1}
\begin{split}
P_\mu(\tau > t) &= \int_0^\infty \int_0^\infty \mu(x) \exp\left(\alpha x - \frac{\alpha^2 t}{2} + \alpha y\right) \frac{1}{\sqrt{2 \pi t}}\left( e^{-\frac{(x-y)^2}{2t}} - e^{-\frac{(x+y)^2}{2t}}\right) dy dx \\
&= \frac{e^{-\frac{\alpha^2 t}{2}}}{\sqrt{2\pi t}} \left(\int_0^\infty \mu(x) e^{-\frac{x^2}{2t}} e^{\alpha x} \int_0^\infty e^{-\frac{y^2}{2t}} e^{-\alpha y} \left( e^{\frac{xy}{t}} - e^{-\frac{xy}{t}}\right) dy dx\right)
\end{split}
\end{equation}

We substitute $z = tx$.
\begin{equation}\label{eq:m2}
\eqref{eq:m1} = \frac{e^{-\frac{\alpha^2 t}{2}}t}{\sqrt{2\pi t}} \left(\int_0^\infty \mu(tz) e^{-\frac{tz^2}{2}}e^{\alpha t z} \int_0^\infty e^{-\frac{y^2}{2t}} e^{-\alpha y} \left( e^{tz} - e^{-tz}\right) dy dz\right)
\end{equation}

For convenience, we will use $x$ instead of $z$ for \eqref{eq:m2} in later parts. \\

From the above equations, the natural construction of $\nu_t$ would come from the terms that consist the outer integral. Indeed, we will use $\displaystyle \nu_t(x) = \mu(tx) e^{-\frac{tx^2}{2}}e^{\alpha tx}$ in Section \ref{sec:medium}. In Section \ref{sec:light}, \eqref{eq:m1} will be used with some modification.
\subsection{Scheffe's Lemma}
From \eqref{eq:master_formula} and \eqref{eq:survival_rate}, we can consider the conditional density
\begin{equation}
\begin{split}
P_\mu(X_t \in dy \; | \; \tau > t) &= \frac{P_\mu(X_t \in dy, \tau > t)}{P_\mu(\tau > t)} \\
&= \frac{\int_0^\infty \mu(x) f(t, x,y) dx}{\int_0^\infty \int_0^\infty \mu(x) f(t, x,y) dy dx}
\end{split}
\end{equation}
When $t$ is fixed, this is clearly a probability density which we will call $\mu_t(y)$. In order to prove convergence of the probability distributions, we will use the following version of Scheffe's Lemma \cite[p. 55]{williams}:
\begin{lemma}
\label{lem:weak_scheffe}
Suppose that $f_n,f$ are probability densities on $\rr_+$ satisfying 
$ \liminf f_n \ge f$, a.e. Then $\int_A  f_n(x)  dx \to \int_A f (x)dx$ for any $A$.
\end{lemma} 
\begin{proof}
Let $d m_n = f_n dx$, and $dm_\infty = f dx$. By Fatou's lemma, for every $A$, 
 \begin{equation}
     \label{eq:myfatou} 
     \liminf  m_n(A)  \ge m_\infty(A)
 \end{equation}

 Now 
 $$1-\limsup  m_n(A) = \liminf  (1- m_n (A)) = \liminf m_n (A^c),$$ 

 Thus, by  \eqref{eq:myfatou} applied to $A^c$,  
 $$ 1- \limsup m_n (A) = \liminf m_n (A^c) \ge  m_\infty(A^c)=1-m_\infty(A).$$ 

 In other words $\limsup m_n (A)\le m_\infty (A)$ and the first statement follows. 
 \end{proof}

%% file: sec4.tex
\section{Proofs: Exponential  or Lighter Tails} \label{light_tail_proofs}

\subsection{Proof of Theorem \ref{thm:light_tails} (Critical and Super-critical Tails)}\label{sec:light}
Throughout this section we will assume that \eqref{eq:crit_or_lighter} holds. 

Define 
$$ f(t, x,y) = y e^{-\alpha y} e^{-\frac{y^2}{2t}}  \frac{\sinh(xy)}{xy}$$ 

and 
$$h(t,x) = \int_0^\infty f(t, x,y) dy$$

and let 
$$h(x) = \lim_{t\to\infty} h(t,x)=\int_0^\infty  e^{-\alpha y} \frac{\sinh(xy)}{x} dy$$ 

Note that $h(x)$ is increasing, 
$$h(0):=\lim_{x\searrow 0} h(x)=\int_0^\infty ye^{-\alpha y} dy= \frac{1}{\alpha^2}$$

and $h(x)=\infty$ if and only if $x\ge \alpha$.  \\

For every $t$, we define two  measures on $[0,\infty)$: 
\begin{equation}
\begin{split}
d \gamma(x)  &= x e^{\alpha x} d\mu(x)\\
d \nu_t (x) &= e^{-\frac{x^2}{2t}} d \gamma (x) 
\end{split} 
\end{equation}

By assumption, there exists a function $\delta(x)\to0$ such that 
$$ \gamma([0,x])\le e^{\delta (x) x }$$ 

without loss of generality, we may also assume $\delta$ is decreasing. \\
 
Observe that 
\begin{equation} 
\label{eq:light_ratio} 
P(X_t \in dy | \tau> t) = \frac{  \int f(t, x/t,y)d \nu_t (x) }{\int h(t,x/t) d\nu_t(x)}.
\end{equation}

We will now prove the theorem through the application of Lemma \ref{lem:weak_scheffe}, where 
$$f_t(v)=  \frac{\int f(t,x/t,y)d \nu_t (x) }{\int h(t,x/t) d\nu_t(x)}$$

and $\displaystyle f(v) = \alpha^2 y e^{-\alpha y}$

\begin{proof}[Proof of Theorem \ref{thm:light_tails}]
Let $\epsilon \in (0,1)$ and let  $\eta_t = \epsilon \alpha t$.  We begin by analyzing the behavior of the denominator in the right-hand side of \eqref{eq:light_ratio}. \\

Observe that $h(t,y)$ is bounded on $[0,M]\times{\mathbb R}_+$  and  increases as $t\to\infty$ to 
$$h(x) = \int_0^\infty y e^{-\alpha y} \frac{\sinh(xy)}{xy} dy$$ 

As a result, the convergence is uniform. From this it follows that
\begin{equation}\label{eq:light_first}
\limsup_{t\to\infty}  \frac{\int_{[0,\eta_t]} h(t,x/t) d \nu_t(x)}{\nu_t([0,\eta_t])} \le  h(\epsilon\alpha).
\end{equation}

We turn to evaluation of the interval on $[\eta_t, 0.9\alpha t]$. Since here $\displaystyle \frac{x}{t} \le 0.9\alpha <\alpha$, $h\left(t,\frac{x}{t}\right)$ is uniformly bounded by a constant depending only on $\alpha$. Below $C$ denotes a positive  constant depending  only on $\alpha,\epsilon$, and whose value may change from line to line. \\

Integrating by parts,
$$ \int_{[\eta_t, 0.9\alpha t]} h\left(t,\frac{x}{t}\right) d \nu_t (x) \le C  \frac{1}{t} \int_{[\eta_t,0.9\alpha t]}x e^{-\frac{x^2}{2t}}  \gamma ([\eta_t,x]) dx.$$ 

Changing variables to $\displaystyle z= \frac{x}{\sqrt{t}}$, the last expression becomes 
$$\int_{\sqrt{t}\alpha  [\epsilon,0.9]}z e^{-\frac{z^2}{2}}  \gamma ([ \eta_t,\sqrt{t}z ]) dz$$ 

Now 
$$\gamma ([\eta_t, \sqrt{t}z]) \le \gamma ([0,\sqrt{t}z]) \le \gamma ([0,\eta_t])e^{\delta (\eta_t) (\sqrt{t}z-\sqrt{t} \epsilon)}\le \gamma ([0,\eta_t])e^{\delta (\eta_t) \sqrt{t}z}$$

Putting this back in the integral gives an upper bound of the form 
 $$ \gamma ([0,\eta_t]) \int_{\sqrt{t}\alpha  [\epsilon,0.9]}z e^{-\frac{z^2}{2}}  e^{\delta(\eta_t) \sqrt{t} z} dz$$
 
Since $\delta(\eta_t)\to 0$ as $t\to\infty$, for all $t$ large enough, we have  
\begin{equation} 
\label{eq:saves_me} \delta(\eta_t) \le \min \left( \frac{\alpha^2\epsilon^2}{4}, \alpha \epsilon\right)
\end{equation}

To obtain an upper bound on the integral, observe that as a function of $z$,  
$$-\frac{z^2}{2} +\delta(\eta_t)\sqrt{t}z= -\frac{z}{2}(z- 2 \delta (\sqrt{t}))$$ 

is decreasing on $[\delta(\eta_t)\sqrt{t},\infty)$, and by \eqref{eq:saves_me},  if $\displaystyle z>\frac{\eta_t}{\sqrt{t}}=\epsilon \alpha \sqrt{t}$, then $z > \delta(\eta_t) \sqrt{t}$.\\

Therefore we have 
\begin{equation}
\begin{split}
-\frac{z^2}{2} + \delta (\eta_t) \sqrt{t} z &\le -\frac{(\eta_t/\sqrt{t})^2}{2} + \delta(\eta_t) \sqrt{t}\left(\frac{\eta_t}{\sqrt{t}}\right) \\
&\le  -\frac{\alpha^2\epsilon^2 t}{2} + \frac{\alpha^2 \epsilon^2 t}{4} \\
&= -\frac{(\alpha \epsilon)^2t}{4}
\end{split}
\end{equation}
 
Thus, 
\begin{equation}
\label{eq:light_second}
    \begin{split} 
  \int_{[\eta_t, 0.9\alpha t]} h\left(t,\frac{x}{t}\right) d \nu_t (x)&\le 
C e^{-\frac{(\alpha\epsilon)^2t}{4}} t^{\frac{3}{2}} \gamma ([0,\eta_t])\\
& \le C e^{ \left(-\frac{(\alpha \epsilon)^2}{4} + \delta(\eta_t)\epsilon \alpha \right)t }t^{\frac{3}{2}}\to 0
\end{split} 
\end{equation} 

Next we consider the behavior over the interval $[0.9\alpha t,\infty)$. Observe that 
$$h(t,x) \le \frac{\sqrt{2\pi t}}{x} E \left[ e^{(x-\alpha) \sqrt{t} Z}\right]$$ 

where $Z$ is standard Gaussian, and therefore 
$$ h\left(t,\frac{x}{t}\right) \le \frac{\sqrt{2\pi t}}{x/t} e^{\frac{x^2}{2t}}e^{\frac{\alpha^2t}{2}} e^{-\alpha x}$$ 

Hence 
$$ \int_{[0.9\alpha t,\infty)} h\left(t,\frac{x}{t}\right) d \nu_t (x) \le 
\sqrt{2\pi t^3 } e^{\frac{\alpha^2t}{2}}\int_{[0.9\alpha t ,\infty)}  d \mu(x)$$ 

But $\mu ([0.9\alpha t,\infty)) = e^{-0.9\alpha^2 t  (1+o(1)) }$, and as a result 
\begin{equation}
    \label{eq:light_third}
 \int_{[0.9\alpha t,\infty)} h\left(t,\frac{x}{t}\right) d \nu_t (x) \to 0.
\end{equation}

Since $\liminf_{t\to\infty} \nu_t ([0,\eta_t])>0$, it follows from \eqref{eq:light_first}, \eqref{eq:light_second}  and \eqref{eq:light_third}, that 
\begin{equation} 
\label{eq:light_denominaor} 
\limsup_{t\to\infty}\frac{\int  h(t,x/t) d \nu_t(x)}{\nu_t([0,\eta_t])} \le h(\epsilon\alpha).
\end{equation}

Repeating the argument leading to  that gave \eqref{eq:light_first} mutatis mutandis, we obtain 
\begin{equation}\label{eq:light_numerator}
\begin{split}
 \liminf_{t\to\infty} \frac{\int_{[0,\eta_t]} f(t,x/t,y) d \nu_t (x)}{\nu_t ([0,\eta_t])} &\ge  y e^{-\alpha y} \inf_{x\le \epsilon \alpha} \frac{\sinh(xy)}{xy}\\
&= y e^{-\alpha y} 
\end{split}
\end{equation} 

It therefore follows from \eqref{eq:light_denominaor} and \eqref{eq:light_numerator}, that 

$$ \liminf_{t\to\infty}\frac{ \int f(t,x/t,y) d \nu_t(x)}{\int h(t,x/t) d\nu_t(x)}\ge \frac{y e^{-\alpha y}}{h(\epsilon\alpha )}$$ 

and this holds for every $\epsilon\in (0,0.9)$.\\

Therefore since $\lim_{\epsilon\to0} h(\epsilon\alpha) =\int_0^\infty y e^{-\alpha y} dy$, we obtain 

$$ \liminf_{t\to\infty} \frac{ \int f (t,x/t,y) d\nu_t (x) }{\int h(t,x/t) d \nu_t (x)} \ge \frac{y e^{-\alpha y}}{\int_0^\infty y e^{-\alpha y}dy}$$ 

and the result follows from Lemma  \ref{lem:weak_scheffe}. 
\end{proof}


\subsection{Proof of Theorem \ref{thm:medium} (Sub-critical Tails)}\label{sec:medium}
Throughout this section, we assume that  \eqref{assume:medium} holds. 

We first split \eqref{eq:m1} into three parts.
\begin{equation}\label{eq:m3}
\begin{split}
P_\mu(\tau > t) &= \frac{e^{-\frac{\alpha^2 t}{2}}}{\sqrt{2 \pi t}} \left ( \underset{=J_3(t)}{\underbrace{\int_0^M e^{-\frac{x^2}{2t}} e^{\alpha x} h\left(t, \frac{x}{t}\right) d\mu(x)}}\right.\\ 
&+ \left. \underset{=J_1(t)}{\underbrace{\int_M^{st} e^{-\frac{x^2}{2t}} e^{\alpha x} h\left(t, \frac{x}{t}\right) d\mu(x)}} +\underset{=J_2(t)}{\underbrace{\int_{st}^\infty e^{-\frac{x^2}{2t}} e^{\alpha x} h\left(t, \frac{x}{t}\right) d\mu(x)}}\right)
\end{split}
\end{equation}

Where $\displaystyle h(t,x) = \int_0^\infty e^{-\frac{y^2}{2t}} e^{-\alpha y} \sinh (xy) dy$.

Here, $M$ is chosen such that we have the following inequality
\begin{equation}\label{eq:medasymp}
\frac{e^{-(\rho+\epsilon)x}}{\rho+\epsilon} \le \frac{\mu([x, \infty))}{c} \le \frac{e^{-(\rho-\epsilon)x}}{\rho-\epsilon}
\end{equation}

For each $x > M$ and some arbitrary $\epsilon > 0$. ($c$ is the normalizing constant of $\mu$) Also, we choose $s$ such that $s = \alpha - \eta$ for some $\alpha > \eta > 0$ that depend on $\mu$. Finally, since we are only interested in the limiting behavior with respect to $t$, we write $M < st$ which is always true for large enough $t$. \\

\begin{proposition}\label{prop:medium1}
Under assumption \eqref{assume:medium}
\begin{equation}
\begin{split}
P_\mu(\tau > t) &\sim \frac{e^{-\frac{\alpha^2 t}{2}}}{\sqrt{2 \pi t}} J_1(t) \\
&\sim c e^{-\frac{(2\alpha \rho - \rho^2) t}{2}} \left(\frac{1}{\rho} - \frac{1}{2\alpha - \rho}\right)
\end{split}
\end{equation}
where $c$ is the constant in \eqref{eq:medasymp} which only depend on $\mu$.
\end{proposition}

\begin{proof}
We first look at the region for $J_1(t)$. In this interval we have the following.

\begin{equation}
\begin{split}
J_1(t) &= \int_M^{st} e^{-x^2/(2t)} e^{\alpha x} h\left(t, \frac{x}{t}\right) d\mu(x) \\
&= t\int_{M/t}^{s} e^{-tx^2/2}e^{\alpha t x} h(t,x) d\mu(tx)
\end{split}
\end{equation}

Some observations on $h(t,x)$ :
\begin{enumerate}
\item $h(t,x)$ is bounded in $\rr_+ \times [0,s]$ since $s < \alpha$.
\item $\displaystyle h(x) = \lim_{t \to \infty} h(t,x) = \frac{1}{\alpha - x} - \frac{1}{\alpha + x}$ by dominated convergence theorem. Moreover, $h(x)$ is also bounded in $[0,s]$.
\end{enumerate}

We introduce a new sequence of measures $(\nu^+_t, \nu^-_t, t \ge 0)$ defined as
\begin{equation}
\begin{split}
d\nu^+_t (x) = e^{-\frac{tx^2}{2}} e^{\alpha t x} e^{-(\rho-\epsilon) t x} &= \sqrt{\frac{2\pi}{t}}e^{\frac{(\alpha - \rho + \epsilon)^2t}{2}} \sqrt{\frac{t}{2\pi}} e^{-\frac{t(x - (\alpha - \rho + \epsilon))^2}{2}}\\
d\nu^-_t (x) = e^{-\frac{tx^2}{2}} e^{\alpha t x} e^{-(\rho+\epsilon) t x} &= \sqrt{\frac{2\pi}{t}}e^{\frac{(\alpha - \rho - \epsilon)^2t}{2}} \sqrt{\frac{t}{2\pi}} e^{-\frac{t(x - (\alpha - \rho - \epsilon))^2}{2}}
\end{split}
\end{equation}

For both case notice that the latter part is a Gaussian density with mean $\alpha - \rho \pm \epsilon$ and variance $1/t$, therefore we have the following convergence of measure:
\begin{equation}\label{eq:medium_delta}
\begin{split}
\nu^+_t &\rightharpoonup \sqrt{\frac{2\pi}{t}}e^{\frac{(\alpha - \rho + \epsilon)^2t}{2}} \delta_{\alpha - \rho + \epsilon}\\
\nu^-_t &\rightharpoonup \sqrt{\frac{2\pi}{t}}e^{\frac{(\alpha - \rho - \epsilon)^2t}{2}} \delta_{\alpha - \rho - \epsilon}
\end{split}
\end{equation}

Therefore, 
\begin{equation}
\begin{split}
\limsup_{t\rightarrow \infty} J_1(t) &= \limsup_{t\rightarrow \infty} c\sqrt{2 \pi t}\int_{M/t}^{s} h(t,x) d\nu^+_t(x)\\
&= c\sqrt{2 \pi t} e^{\frac{(\alpha - \rho + \epsilon)^2t}{2}} \left(\frac{1}{\rho-\epsilon} - \frac{1}{2\alpha - \rho + \epsilon}\right)
\end{split}
\end{equation}

\begin{equation}
\begin{split}
\liminf_{t\rightarrow \infty} J_1(t) &= \liminf_{t\rightarrow \infty} c\sqrt{2\pi t}\int_{M/t}^{s} h(t,x) d\nu^-_t(x)\\
&= c\sqrt{2 \pi t} e^{\frac{(\alpha - \rho - \epsilon)^2t}{2}} \left(\frac{1}{\rho+\epsilon} - \frac{1}{2\alpha - \rho - \epsilon}\right)
\end{split}
\end{equation}

and since $\epsilon$ is arbitrary, we conclude that
\begin{equation}
J_1(t) \sim c\sqrt{2 \pi t}e^{\frac{(\alpha - \rho)^2t}{2}} \left(\frac{1}{\rho} - \frac{1}{2\alpha - \rho}\right)
\end{equation}

For the second interval $x \in (st, \infty)$ we first study some bound for $h(t,x/t)$. we start from the obvious.
\begin{equation} h\left(t, \frac{x}{t}\right) \le \int_0^\infty \exp\left(-\frac{y^2}{2t} + \alpha y + \frac{xy}{t}\right) \end{equation}

We can rewrite the exponent as 
\begin{equation}
\begin{split}
 - \frac{y}{2\sqrt{t}}\left( \frac{y}{\sqrt{t}}+2 \alpha \sqrt{t} -\frac{2x}{\sqrt{t}}\right)&=-\frac{1}{2} \frac{y}{\sqrt{t}} \left(\frac{y}{\sqrt{t}}+2\varphi\right)\\
  & = -\frac12 (w - \varphi )(w+\varphi)
 \end{split} 
 \end{equation} 

where $\displaystyle \varphi=  \left(\sqrt{t}\alpha - \frac{x}{\sqrt{t}}\right)$, and $\displaystyle w= \frac{y}{\sqrt{t}} + \varphi$. Therefore, after changing  variables $y\to w$, we obtain 
\begin{equation}\label{h_bound}
\begin{split}
h(t,x) &\le \sqrt{t}e^{\frac{\varphi^2}{2}} \int_\varphi^\infty e^{-\frac{w^2}{2}}dw\\
  & = \sqrt{t} e^{\frac{\alpha^2 t}{2}} e^{\frac{x^2}{2t}}  e^{-\alpha x} L \left(\sqrt{t}\alpha - \frac{x}{\sqrt{t}}\right),
 \end{split}
 \end{equation}

 where $\displaystyle L(z) = \int_z^\infty e^{-\frac{w^2}{2}} dw$.\\

$L$ has some nice properties:
\begin{enumerate}
\item $L(z)$ is strictly decreasing and bounded above by $\sqrt{2\pi}$.
\item When $z$ is negative, $L(z) < \sqrt{2 \pi}$.
\item When $z$ is positive, \begin{equation}L(z) \le \min\left(\frac{e^{-\frac{z^2}{2}}}{z}, \sqrt{\frac{\pi}{2}}\right)\end{equation}
\item More specifically, if $z \ge 1$ then \begin{equation}\label{standardbound} L(z) \le e^{-\frac{z^2}{2}}\end{equation}
\end{enumerate}

Using the bound above we get the following.
\begin{equation}
\begin{split}
J_2(t) &\le \sqrt{t} \int_{st}^\infty e^{\frac{\alpha^2 t}{2}}  L \left(\sqrt{t}\alpha - \frac{x}{\sqrt{t}}\right) d\mu(x) \\
&\le c \sqrt{2 \pi t} e^{\frac{\alpha^2 t}{2}} e^{-\rho s t} \\
&= c\sqrt{2 \pi t} e^{t\left(\frac{\alpha^2}{2} - \rho(\alpha - \eta)\right)}
\end{split}
\end{equation}

We want $\displaystyle J_2(t) = o(J_1(t)) = o\left(\sqrt{t}e^{\frac{(\alpha-\rho)^2 t}{2}}\right)$. Indeed, if we pick $\eta = \rho/4$,
\begin{equation}
\begin{split}
\frac{(\alpha - \rho)^2}{2} - \left(\frac{\alpha^2}{2} - \gamma(\alpha - \eta)\right) &= \frac{\rho^2}{2} - \rho\eta \\
&= \frac{\rho^2}{4} > 0
\end{split}
\end{equation}

therefore we get the desired asymptotic. \\

For the last interval $x \in [0,M]$, we use the fact that for any $\epsilon > 0$, we can fix $t_0$ such that for each $t > t_0$, $M/\sqrt{t} < \epsilon$. And for such $t$, we have
\begin{equation}\label{eq:medj31}
\begin{split}
J_3(t) &= \int_0^M e^{-\frac{x^2}{2t}} e^{\alpha x} \int_0^\infty e^{-\frac{y^2}{2t}} e^{-\alpha y} \sinh\left(\frac{xy}{t}\right) dy d\mu(x) \\
& \le \sqrt{t} e^{\frac{\alpha^2 t}{2}} \int_0^M \mu(x) L\left(\sqrt{t}\alpha - \frac{x}{\sqrt{t}}\right) d\mu(x)
\end{split}
\end{equation}
And since $L$ is decreasing,
\begin{equation}\label{eq:medj32}
\eqref{eq:medj31} \le \sqrt{t} e^{\frac{\alpha^2 t}{2}} \int_0^M L\left(\sqrt{t}\alpha - \epsilon \right) d\mu(x)
\end{equation}
Finally using \eqref{standardbound} and that $\mu$ is a probability measure,
\begin{equation}
\begin{split}
\eqref{eq:medj32} & \le \sqrt{t} \int_0^M e^{\alpha\epsilon\sqrt{t}}e^{-\frac{\epsilon^2}{2}} d\mu(x) \\
& \le \sqrt{t} e^{\alpha\epsilon\sqrt{t} - \frac{\epsilon^2}{2}}\\
&= o\left(\sqrt{t}e^{\frac{(\alpha-\rho)^2t}{2}}\right) = o(J_1(t))
\end{split}
\end{equation}
\end{proof}

We now turn to computing the limiting density.

\begin{equation}
\begin{split}
P_\mu(X_t \in dy, \tau>t) &= \frac{e^{-\frac{\alpha^2 t}{2}}}{\sqrt{2\pi t}} \left(\int_0^\infty \underset{=g(x,y,t)}{\underbrace{e^{-\frac{x^2}{2t}} e^{\alpha x} e^{-\frac{y^2}{2t}} e^{-\alpha y} \sinh\left(\frac{xy}{t}\right) }} d\mu(x) \right)\\
&= \frac{e^{-\frac{\alpha^2 t}{2}}}{\sqrt{2\pi t}} \left(\underset{= K_3(t, y)}{\underbrace{\int_0^M g(x,y,t) d\mu(x)}} + \underset{= K_1(t, y)}{\underbrace{\int_M^{st}g(x,y,t) d\mu(x)}} + \underset{= K_2(t, y)}{\underbrace{\int_{st}^\infty g(x,y,t) d\mu(x)}} \right)
\end{split}
\end{equation}

Where $M, s$ are the same as \eqref{eq:m3}. 

\begin{proposition}\label{prop:medium2}
Under assumption \eqref{assume:medium}, 
\begin{equation}
\begin{split}
P_\mu(X_t \in dy, \tau > t) &\sim \frac{e^{-\frac{\alpha^2 t}{2}}}{\sqrt{2\pi t}} K_1(t,y) \\
&\sim c e^{-\frac{(2\alpha \rho - \rho^2)t}{2}} e^{-\alpha y} \sinh ((\alpha - \rho) y)
\end{split}
\end{equation}
where $c$ is the constant in \eqref{eq:medasymp} which only depends on $\mu$.
\end{proposition}

\begin{proof}
Using similar estimation method and sequence of measures $(\nu_t^+, \nu_t^-, t \ge 0)$ as before, we can see that for each $y \in \rr_+$

\begin{equation}
\begin{split}
\limsup_{t \to \infty} K_1(t,y) &= \limsup_{t \to \infty} c\sqrt{2\pi t} \int_{M/t}^s e^{-\frac{y^2}{2t}}e^{-\alpha y} \sinh (xy) d\nu_t^+(x) \\
&= c\sqrt{2\pi t} e^{\frac{(\alpha - \rho + \epsilon)^2 t}{2}} e^{-\alpha y} \sinh ((\alpha - \rho + \epsilon) y)
\end{split}
\end{equation}

\begin{equation}
\begin{split}
\liminf_{t \to \infty} K_1(t,y) &= \liminf_{t \to \infty} c\sqrt{2\pi t} \int_{M/t}^s e^{-\frac{y^2}{2t}}e^{-\alpha y} \sinh (xy) d\nu_t^-(x) \\
&= c\sqrt{2\pi t} e^{\frac{(\alpha - \rho - \epsilon)^2 t}{2}} e^{-\alpha y} \sinh ((\alpha - \rho - \epsilon) y)
\end{split}
\end{equation}

and therefore

\begin{equation}
K_1(t,y) \sim c\sqrt{2\pi t} e^{\frac{(\alpha - \rho)^2 t}{2}} e^{-\alpha y} \sinh ((\alpha - \rho) y)
\end{equation}

For $K_2(t,y)$ we use the upper bound in \eqref{eq:medasymp} to get the following estimate.

\begin{equation}\label{eq:medj21}
\begin{split}
K_2(t, y) &\le e^{-\frac{y^2}{2t}} e^{-\alpha y} \int_{st}^\infty e^{-\frac{x^2}{2t}} e^{(\alpha-\rho+\epsilon)x} e^{\frac{xy}{t}} dx \\
& = \sqrt{t} e^{\frac{(\alpha-\rho+\epsilon)^2t}{2}} e^{(-\rho+\epsilon)y} L\left(\sqrt{t}s - \sqrt{t}(\alpha-\rho+\epsilon) - \frac{y}{\sqrt{t}}\right)
\end{split}
\end{equation}

Since $s - (\alpha -\rho + \epsilon) > 0$ for small enough $\epsilon$, the argument for $L$ above is strictly positive and increasing. Therefore by \eqref{standardbound},

\begin{equation}\label{eq:medj22}
\eqref{eq:medj21} \le \sqrt{t}e^{\frac{(\alpha - \rho)^2 t}{2}}\exp\left(- \frac{(s - (\alpha - \rho))^2 t}{2} + (2(\alpha - \rho) - s) \epsilon t\right) e^{(s - \alpha + 2\epsilon) y}
\end{equation}

Again, $s - (\alpha - \rho) > 0$ and $\epsilon$ is arbitrarily small so the middle term above is exponentially decaying. We conclude that

\begin{equation}
\eqref{eq:medj22} = o\left(\sqrt{t}e^{\frac{(\alpha - \rho)^2 t}{2}} \right) = o(K_1(t,y))
\end{equation}

Finally for $K_3(t,y)$ we can directly apply the dominated convergence theorem.

\begin{equation}
\begin{split}
K_3(t,y) &\sim \int_0^M e^{\alpha x} e^{-\alpha y} \sinh(0) d\mu(x) \\
&= o(1) = o(K_1(t,y))
\end{split}
\end{equation}

\end{proof}

We can now prove Theorem \ref{thm:medium}.

\begin{proof}[Proof of Theorem \ref{thm:medium}]
The fact that $\epsilon$ is arbitrarily small in \eqref{eq:medium_delta} proves the first part of the theorem. The second part follows from  Proposition \ref{prop:medium1} and  Proposition \ref{prop:medium2} and Lemma  \ref{lem:weak_scheffe}. 
\end{proof}

%% file: sec5.tex
\section{Proofs: Heavy Tails}\label{heavy_tail_proofs}
In Section Section \ref{sec:51} we will prove Lemma \ref{lem:heavy_tail} to see that adequate scaling is necessary to obtain a non-trivial quasi-limiting behavior. In Section \ref{sec:52} we obtain the scaling by    estimating the tail distribution of the surviving process, Proposition \ref{subexprop}.  In Section \ref{sec:53} we will use this  to prove Theorem \ref{thm:subexp}. Finally, in Section \ref{sec:examples} we present a number of concrete applications to Theorem \ref{thm:subexp}. 
\subsection{On tail of the initial distribution and tail of survival time}
\label{sec:51} 
\begin{proof}[Proof of Lemma \ref{lem:heavy_tail}]
Pick $b > 0$ such that $\displaystyle \sinh (\alpha b) > \frac{1}{4}e^{\alpha b}$ then by Proposition \ref{prop:masterformula} we have:
\begin{equation*}
\begin{split}
P_\mu(\tau > t) &\ge P_\mu(X_0 > \alpha t, X_t > b, \tau>t)\\
&= \frac{e^{-\frac{\alpha^2 t}{2}}t}{\sqrt{2\pi t}} \int_{\alpha}^\infty \mu(tx) e^{-\frac{tx^2}{2}} e^{\alpha t x} \int_b^\infty e^{-\frac{y^2}{2t}} e^{-\alpha y} (e^{xy} - e^{-xy}) dy dx \\
&\ge \frac{e^{-\frac{\alpha^2 t}{2}}t}{4\sqrt{2\pi t}} \int_{\alpha}^\infty \mu(tx) e^{-\frac{tx^2}{2}} e^{\alpha t x} \int_b^\infty e^{-\frac{y^2}{2t}} e^{-\alpha y} e^{xy} dy dx \\
&= \frac{t}{4\sqrt{2\pi}} \int_{\alpha}^\infty \mu(tx) L\left(\frac{b}{\sqrt{t}} + \sqrt{t}(\alpha-x)\right) dx \\
&\ge \frac{1}{8}\mu([t\alpha, \infty))
\end{split}
\end{equation*}
This implies that $P_\mu(\tau > t)$ is at least as heavy as the tail distribution of $\mu$. By Proposition \ref{prop:heavytail}, any initial distribution $\mu$ that has heavier-than-exponential tail distribution cannot converge to a QSD.
\end{proof}

\subsection{Distribution of the surviving processes}\label{sec:52}

The method we develop here works for a large class of distributions $\mu$, yet both scaling and limit distributions may depend on the choice of $\mu$. \\

Recall that we work under the Assumption \ref{assume:nice_heavy_tail}. We can write the density of $\mu$ as follows.
\begin{equation}\label{density}
\text{If } \beta > 0 \text{ then } \mu(x) = F'(x)\exp(-F(x)) = F'(x)\mu([x,\infty))
\end{equation}

Note that by \cite[Proposition 1.8.1]{rv_bingham} , $F'(x)$ is smooth varying with index $\beta-1$. \\

We turn to the tail distribution. By the above assumption, $\mu$ has a continuous density, which we also denote by $\mu$.

\begin{equation}
\begin{split}
&P_\mu(X_t > a_t, \tau > t) 
\begin{aligned}[t]&=\frac{e^{-\frac{\alpha^2 t}{2}}}{\sqrt{2\pi t}} \left(\int_0^\infty e^{-\frac{x^2}{2t}} e^{\alpha x} \int_{a_t}^\infty e^{-\frac{y^2}{2t}} e^{-\alpha y} \left(e^{\frac{xy}{t}} - e^{-\frac{xy}{t}}\right) dy d\mu(x)\right) \\
&= \frac{e^{-\frac{\alpha^2 t}{2}}t}{\sqrt{2\pi t}} \int_0^\infty \mu(tx) e^{-\frac{tx^2}{2}} e^{\alpha t x} \int_{a_t}^\infty e^{-\frac{y^2}{2t}} e^{-\alpha y} (e^{xy} - e^{-xy}) dy dx
\end{aligned} \\
&= \frac{t}{\sqrt{2\pi}} \left(\underset{=J_1(t)}{\underbrace{\int_0^\infty \mu(tx) L\left(\frac{a_t}{\sqrt{t}} + \sqrt{t}\alpha - \sqrt{t}x\right) dx}}\right. - \left.\underset{= J_2(t)}{\underbrace{\int_0^\infty \mu(tx) e^{2\alpha t x} L\left(\frac{a_t}{\sqrt{t}} + \sqrt{t}\alpha + \sqrt{t}x\right) dx}} \right)
\end{split}
\end{equation}

We first notice that from the second term $J_2$,
\begin{equation}
\begin{split}
e^{2\alpha t x} L\left(\frac{a_t}{\sqrt{t}} + \sqrt{t}\alpha + \sqrt{t}x\right) dx &\le e^{2 \alpha t x} \frac{e^{-\frac{a_t^2/t + t\alpha^2 + tx^2 + 2 a_t \alpha + 2 a_t x + 2 \alpha t x}{2}}}{a_t/\sqrt{t} + \sqrt{t}\alpha + \sqrt{t}x} \\
&= \frac{e^{-\frac{t(\alpha - x)^2}{2}}e^{-a_t^2/(2t)} e^{-a_t(\alpha+x)}}{a_t/\sqrt{t} + \sqrt{t}\alpha + \sqrt{t}x}
\end{split}
\end{equation}

If $a_t \gg \epsilon\sqrt{t}$ then the term $e^{-a_t^2/(2t)}$ will let $J_2$ decay faster (in exponential sense) than $\mu(tx)$. In fact, unless $a_t = o(\sqrt{t})$ and $x \in (\alpha - t^{-1/2 + \epsilon}, \alpha + t^{-1/2 + \epsilon})$, $J_2$ decays exponentially faster than $\mu(tx)$. \\

Furthermore, when we define $J_{1,A}(t), J_{2,A}(t)$ to be integrated over some sub-interval $A$ of $\mathbb{R}_+$ instead of the entire $\mathbb{R}_+$ as follows:
\begin{equation}
\begin{split}
J_{1,A}(t) &= \int_A \underset{= f(t,x)}{\underbrace{\mu(tx) L\left(\frac{a_t}{\sqrt{t}} + \sqrt{t}\alpha - \sqrt{t}x\right)}} dx\\
J_{2,A}(t) &= \int_A \mu(tx) e^{2\alpha t x} L\left(\frac{a_t}{\sqrt{t}} + \sqrt{t}\alpha + \sqrt{t}x\right) dx
\end{split}
\end{equation}
since $P_\mu(X_0 \in \cdot, X_t \in \cdot, \tau > t) \geq 0$ always, we can claim that $J_{2,A} = O(J_{1,A})$ on the same sub-interval $A \in \mathbb{R}_+$.

For the first term $J_1$, we split the integration.
\begin{equation}\label{j1_estimate}
\begin{split}
J_1(t) &= \underset{= J_{1,1}(t)}{\underbrace{\int_0^{\alpha + a_t/t - \eta_t} f(t,x)dx}} + \underset{=J_{1,2}(t)}{\underbrace{\int_{\alpha + a_t/t - \eta_t}^{\alpha + a_t/t + \epsilon_t} f(t,x) dx}} + \underset{=J_{1,3}(t)}{\underbrace{\int_{\alpha + a_t/t + \epsilon_t}^\infty f(t,x) dx}}
\end{split}
\end{equation}
where $\eta_t, \epsilon_t$ is to be picked depending on $\mu$. \\

The goal now is to get an accurate asymptotic on the survival rate.

\begin{proposition}\label{subexprop0}
Suppose $\mu$ satisfies Assumption \ref{assume:nice_heavy_tail}. Then for any $\displaystyle \eta_t \gg t^{\beta-1}$,
\begin{equation}
\log J_{1,1}(t) \ll \log \mu([t\alpha + a_t, \infty))
\end{equation}
\end{proposition}

\begin{proof}
Suppose $\displaystyle \eta_t \gg t^{\frac{\beta-1}{2}}$. Then we have the following estimate.
\begin{equation}\label{base}
\begin{split}
J_{1,1}(t) &\le L(\sqrt{t} \eta_t)\int_0^{\alpha + a_t/t-\eta_t} \mu(tx)dx\\
&\le L(\sqrt{t}\eta_t)\\
&\le \exp\left(-\frac{t(\eta_t)^2}{2}\right)\\
&\ll \exp\left(\frac{t^{-\beta}}{2}\right) \sim \mu([t\alpha + a_t, \infty))
\end{split}
\end{equation}

Now suppose $\displaystyle t^{\frac{\beta-1}{2}} \gg \eta_t \gg t^{\beta-1}$. Pick $\displaystyle t^{(\beta-1)/2} \ll \eta_t^1 = t^{r_1}$ such that by \eqref{base},
$$\underset{=J_{1,1,1}(t)}{\underbrace{\int_0^{\alpha + a_t/t -\eta_t^1} \mu(tx)L\left(\sqrt{t}\alpha + \frac{a_t}{\sqrt{t}} - \sqrt{t}x\right) dx}} \ll \mu([t\alpha + a_t, \infty))$$

Now we want to pick $\displaystyle t^{r_2} = \eta_t^2 \ll \eta_t^1$ such that
$$\underset{=J_{1,1,2}(t)}{\underbrace{\int_{\alpha + a_t/t -\eta_t^1}^{\alpha + a_t/t -\eta_t^2} \mu(tx)L\left(\sqrt{t}\alpha + \frac{a_t}{\sqrt{t}} - \sqrt{t}x\right) dx}} \ll \mu([t\alpha + a_t, \infty))$$

Using integration by parts, 
\begin{equation}\label{intparts}
\begin{split}
J_{1,1,2}(t) &= \int_{\alpha + a_t/t -\eta_t^1}^{\alpha + a_t/t -\eta_t^2} \mu(tx)L\left(\sqrt{t}\alpha + \frac{a_t}{\sqrt{t}} - \sqrt{t}x\right) dx\\
&= -\left.\frac{1}{t} \mu([tx, \infty)) L\left(\sqrt{t}\alpha + \frac{a_t}{\sqrt{t}} - \sqrt{t}x\right) \right|_{\alpha + a_t/t -\eta_t^1}^{\alpha + a_t/t -\eta_t^2} \\
&\quad \quad + \frac{1}{\sqrt{t}}\int_{\alpha + a_t/t -\eta_t^1}^{\alpha + a_t/t -\eta_t^2} \exp\left(\frac{-t(x-(\alpha + a_t/t))^2}{2}\right) \mu([tx, \infty)) dx\\
&\le \frac{1}{t}\left(-\mu(t\alpha + a_t - t\eta_t^2, \infty)L(\sqrt{t}\eta_t^2) + \mu(t\alpha + a_t - t\eta_t^1, \infty)L(\sqrt{t}\eta_t^1)\right) \\
&\quad \quad + \frac{1}{t}\mu(t\alpha + a_t - t\eta_t^1, \infty)L(\sqrt{t}\eta_t^2)\\
\end{split}
\end{equation}

Since both $\mu([x,\infty))$ and $L(x)$ are decreasing function, the driving term of \eqref{intparts} is the last one. And since $\displaystyle \mu(x, \infty) = \exp(-F(x))$ where $F$ is an increasing regularly varying function with index $\beta$,

\begin{equation}
\begin{split}
\frac{1}{t}\mu(t\alpha + a_t - t\eta_t^1, \infty)L(\sqrt{t}\eta_t^2) &\sim \frac{1}{t}\mu(t\alpha + a_t - t\eta_t^1, \infty)L(\sqrt{t}\eta_t^2)\\
&\le \frac{1}{t} \exp\left(-F((t\alpha + a_t) - t^{1+r_1})\right) \exp\left(-\frac{t^{1+2r_2}}{2}\right)\\
&\sim \frac{1}{t} \mu(t\alpha + a_t,\infty) \exp\left(t^{\beta + r_1} - \frac{t^{1+2r_2}}{2}\right)
\end{split}
\end{equation}

If $\beta + r_1 < 1+2r_2$ we get the desired asymptotic. That is, we need $\displaystyle r_2 > \frac{(\beta-1)+r_1}{2}$, and combining with $t^{(\beta-1)/2} \ll \eta_t^1$ we can pick
$$\eta_t^2 \gg t^{\frac{(\beta-1)+r_1}{2}} \sim t^{\frac{3(\beta-1)}{4}}$$

to get 
\begin{equation}
\begin{split}
\int_0^{\alpha + a_t/t -\eta_t^2} \mu(tx)L\left(\sqrt{t}\alpha + \frac{a_t}{\sqrt{t}} - \sqrt{t}x\right) dx &= J_{1,1,1}(t) + J_{1,1,2}(t) \\
&\ll \mu([t\alpha + a_t, \infty))
\end{split}
\end{equation}

Recursively, we can pick $\eta_t^n \gg t^{(\beta-1)(1-(1/2)^n)}$ such that
$$J_{1,1,n}(t) = \int_{\alpha + a_t/t -\eta_t^{n-1}}^{\alpha + a_t/t -\eta_t^n} \mu(tx)L\left(\sqrt{t}\alpha + \frac{a_t}{\sqrt{t}} - \sqrt{t}x\right) dx \ll \mu([t\alpha + a_t, \infty))$$

So for sufficiently large $n$ we have
$$\eta_t = \eta_t^n \gg t^{(\beta-1)(1-(1/2)^n)} \gg t^{\beta-1}$$
$$J_{1,1}(t) = \sum_{i=1}^n J_{1,1,i}(t) \ll \mu([t\alpha + a_t, \infty))$$

which completes the proof.
\end{proof}

\begin{proposition}\label{subexprop}
Suppose $\mu$ satisfies \ref{assume:nice_heavy_tail} and $\beta > 0$. If $a_t \gg \sqrt{t}$,
\begin{equation}P_\mu(X_t > a_t, \tau > t) \sim \frac{t}{\sqrt{2\pi}} J_{1,3}(t) \sim \mu([t\alpha+a_t, \infty))\end{equation}
\end{proposition}

\begin{proof}
Pick $\eta_t$ and $\epsilon_t$ as follows.
\begin{equation}\label{betacond} t^{\beta-1} \ll \eta_t \ll 1, \quad \quad \epsilon_t = t^{-b}, \quad \quad \beta < b < 0.5\end{equation}

This choice yields the following asymptotic.
\begin{equation} \label{densityineq}
\eta_t \to 0, \quad \eta_t \ll a_t/t, \quad \epsilon_t \ll a_t/t, \quad \sqrt{t}\epsilon_t \to \infty, \quad F'(t\alpha+a_t)\epsilon_t \ll 1/t
\end{equation}

For $J_{1,2}(t)$, we first observe that the interval $(\alpha + a_t/t - \eta_t, \alpha + a_t/t + \epsilon_t)$ close in to $\alpha + a_t/t$. Moreover, while $L$ does vary between $0$ and $\sqrt{\pi/2}$ within the interval, $\mu$ does not vary much from $\mu(t(\alpha+a_t/t))$ inside the interval, and therefore we can use the intermediate value theorem. Also, we split the integration to get the following bound for $J_{1,2}(t)$.

\begin{align*} 
J_{1,2}(t) &\sim \mu(t(\alpha+a_t/t))\\
 & \quad\quad \times \left(\int_{\alpha + a_t/t - \eta_t}^{\alpha + a_t/t} L\left(\frac{a_t}{\sqrt{t}} + \sqrt{t}\alpha - \sqrt{t}x\right) dx + \int_{\alpha + a_t/t}^{\alpha + a_t/t + \epsilon_t} L\left(\frac{a_t}{\sqrt{t}} + \sqrt{t}\alpha - \sqrt{t}x\right) dx\right) \\
 &\le \mu(t\alpha + a_t) \left(\frac{1}{\sqrt{t}} \int_0^{\sqrt{t}\eta_t} L(y) dy + \int_0^{\epsilon_t} \sqrt{2\pi} dx \right)\\
 &\le \mu(t\alpha + a_t) \left(\frac{1}{\sqrt{t}} \int_0^\infty L(y)dy + \sqrt{2\pi} \epsilon_t \right)\\
 &\sim \sqrt{2\pi} \mu(t\alpha + a_t)\epsilon_t
\end{align*} 

Note that the first integration is essentially the expected value of a half-normal distribution, and second integration is estimated using the fact that $L$ is bounded above.

To estimate $J_{1,3}(t)$, since $\sqrt{t}\epsilon_t \to \infty$, it follows that $L(\sqrt{t}\epsilon_t) \to \sqrt{2\pi}$ and we can use IVT to get the sharp estimate.
\begin{equation}
\begin{split}
J_{1,3}(t) &\sim \sqrt{2\pi} \int_{\alpha + a_t/t + \epsilon_t}^\infty \mu(tx)  dx\\
&\sim \sqrt{2\pi} \frac{1}{t} \mu([t\alpha + a_t, \infty))
\end{split}
\end{equation}

Proposition \ref{subexprop0} shows that $J_{1,1}(t) = o(J_{1,3}(t))$.

For $J_{1,2}(t)$, we combine \eqref{density} and \eqref{densityineq} to get the following asymptotic comparison.
\begin{equation}
\begin{split}
J_{1,2}(t) &\le \sqrt{2\pi}\mu(t\alpha + a_t)\epsilon_t \\
&\ll \frac{\sqrt{2\pi}}{t} \mu([t\alpha + a_t, \infty)) \sim J_{1,3}(t)
\end{split}
\end{equation}

Finally from the choice of $\epsilon_t$ we have $b < 0.5$, and therefore
\begin{equation}
\begin{split}
J_{2,3}(t) &\le \int_{\alpha + a_t/t + \epsilon_t}^\infty \mu(tx) e^{-\frac{t(\alpha - x)^2}{2}}e^{-a_t^2/(2t)} e^{-a_t(\alpha+x)} dx \\
&\le e^{-\frac{t\epsilon_t^2}{2}} \mu([t\alpha + a_t, \infty)) = o(J_{1,3}(t))
\end{split}
\end{equation}

We conclude that
\begin{equation} P_\mu(X_t > a_t, \tau > t) \sim (1+o(1))\frac{t}{\sqrt{2\pi}} J_{1,3}(t) \sim\mu([t\alpha + a_t, \infty)) \end{equation}

\end{proof}

We can extend this proposition to the cases where $F$ is slowly varying. In such cases, we expect the tail distribution $\mu(x, \infty)$ itself to be smoothly varying.

\begin{corollary}\label{rvprop}
Suppose $\mu([x, \infty)) = G(x)$, where $G$ is smoothly varying function with index $-\kappa < 0$. Then $\displaystyle P_\mu(X_t > a_t, \tau > t) \sim \frac{t}{\sqrt{2\pi}} J_{1,3}(t) \sim \mu([t\alpha+a_t, \infty))$.
\end{corollary}

\begin{proof}
It suffices to show that $J_{1,2}(t) = o(J_{1,3}(t))$. The smooth varying condition yields the following relation \cite[1.8.1']{rv_bingham}
\begin{equation}t\mu(t\alpha + a_t) \sim \mu([t\alpha+a_t, \infty))\end{equation}
Since we have $\epsilon_t \ll 1$, 
\begin{equation}J_{1,2}(t) \le \mu(t\alpha + a_t) \epsilon_t = o\left(\frac{1}{t}\mu([t\alpha + a_t, \infty))\right) = o(J_{1,3}(t))\end{equation}
so we have the desired asymptotic.
\end{proof}


\subsection{Proof of Theorem \ref{thm:subexp} (Heavy-tailed  initial distributions)}\label{sec:53}
Proposition \ref{subexprop} and Corollary \ref{rvprop} show why the second part of Assumption \ref{assume:nice_heavy_tail} is necessary. We need the right $a_t$ that will yield nontrivial result on the limit
\begin{equation}\lim_{t \to \infty} P_\mu(X_T > a_t \; | \; \tau > t) = \lim_{t \to \infty} \frac{P_\mu(X_t > a_t, \tau>t)}{P_\mu(\tau > t)}\end{equation}

Due to Proposition \ref{subexprop} this boils down to comparing $\mu(t\alpha,\infty)$ and $\mu(t\alpha + a_t,\infty)$.

\begin{proof}[Proof of Theorem \ref{thm:subexp}]

If $\mu$ satisfies Assumption \ref{assume:nice_heavy_tail}, setting $a_t = R(t,c)$ gives the following.
\begin{equation}\label{tailscale}
\begin{split}
\mu([t\alpha + a_t, \infty)) &= \exp(-F(t\alpha + R(t,c)) \\
&\sim \exp(-(F(t\alpha) + c)) \\
&= e^{-c} \mu(t\alpha, \infty)
\end{split}
\end{equation}

We make few comments on the observation \eqref{taylorcondition}. If smooth enough, $F$ has the Taylor expansion
$$F(t\alpha + R(t,c)) = F(t\alpha) + F'(t\alpha) R(t,c) + o(F'(t))$$

therefore by choosing $\displaystyle R(t,c) = \frac{c}{F'(t\alpha)}$, we get $F(t\alpha + R(t,c)) - F(t\alpha) = c + o(F'(t))$. Since $F$ has index $\beta < 1$, $F'(t) = o(1)$ so condition \eqref{eq:nice_difference} is satisfied.

We further observe that with the choice $\displaystyle R(t,c) = \frac{c}{F'(t\alpha)}$,
\begin{equation}
\begin{split}
F'(t\alpha + R(t,c)) &= F'(t\alpha) + F''(t\alpha)R(t,c) + o(F''(t)) \\
&= F'(t\alpha) + \frac{cF''(t\alpha)}{F'(t\alpha)} + o(F''(t)) \\
&= F'(t\alpha) + o(1)
\end{split}
\end{equation}

Therefore we get $F'(t\alpha + R(t,c)) \sim F'(t\alpha)$, and consequently,
\begin{equation}\label{densityscale}
\begin{split}
\mu(t\alpha + R(t,c)) &= F'(t\alpha + R(t,c))\exp(-F(t\alpha + R(t,c)) \\
&\sim F'(t\alpha)\exp(-(F(t\alpha)+c)) \\
&= e^{-c}\mu(t\alpha)
\end{split}
\end{equation}

Putting together Proposition \ref{subexprop}, Corollary \ref{rvprop}, \eqref{tailscale}, and \eqref{densityscale} completes the proof.
\end{proof}
\subsection{Some Concrete Examples} 
\label{sec:examples}
We present some concrete results here. 

\begin{corollary}
Suppose $\mu([x,\infty)) = e^{-x^\beta}$ with $\beta \in (0,0.5)$. Then
\begin{equation}
\lim_{t \to \infty} P_\mu \left(\left. \frac{X_t}{t^{1-\beta}} > c \; \right| \; \tau > t \right) = \exp(-\beta\alpha^{\beta-1}c)
\end{equation}
that is, the limiting distribution is exponential with parameter $\beta\alpha^{\beta-1}$.
\end{corollary}

\begin{proof}
From proposition \ref{subexprop} we get
\begin{equation}
P_\mu(X_t > a_t, \tau > t) \sim \mu([t\alpha + a_t, \infty)) 
\end{equation}

Pick $a_t = c\cdot t^{1-\beta}$. Then by the generalized binomial theorem,
$$(t\alpha+a_t)^\beta = (t\alpha)^\beta + c\beta\alpha^{\beta-1} + o(1)$$

Note that $F'(t\alpha) = \beta(t\alpha)^{\beta-1}$. By substituting $\displaystyle \overline{c} = c t^{1-\beta} ((\alpha t)^\beta)' = c \beta \alpha^{\beta-1}$, Theorem \ref{thm:subexp} gives us the desired result.
\end{proof}

\begin{ex}
\label{xmpl:weibull} 
If $\mu$ is a Weibull distribution with scale parameter $\lambda >0$ and shape parameter $0 < \beta < 0.5$, the limiting distribution of $\displaystyle P_\mu \left(\left. \frac{X_t}{t^{1-\beta}} > c \; \right| \; \tau > t \right)$ is exponential distribution with rate $\displaystyle \beta \left(\frac{\alpha}{\lambda}\right)^{\beta-1}$.
\end{ex}

\begin{corollary}\label{thm:rv}
Suppose $\mu([x, \infty)) = G(x)$, where $G$ is smoothly varying function with index $-\kappa < 0$. Then 
\begin{equation}
\lim_{t \to \infty} P_\mu\left(\left. \frac{X_t}{t}  > c \; \right| \; \tau > t\right) = \left(\frac{\alpha+c}{\alpha}\right)^{-\kappa}
\end{equation}
that is, the limiting distribution is Lomax (shifted Pareto) distribution with shape parameter $\kappa$ and scale parameter $\alpha$.
\end{corollary}

\begin{proof}
Since $G(x) = \exp(\log(G(x)))$ and $\log(G(x))$ is a slowly varying function ($\beta = 0)$, the natural choice for $R(t,c)$ would be $a_t = R(t,c) = tc$. Indeed, by the uniform convergence theorem of regular varying function, \cite[Theorem 1.5.2]{rv_bingham}
\begin{equation}
\lim_{t \to \infty} \frac{G(t\alpha + tc)}{G(t)} = (\alpha + c)^{-\kappa}
\end{equation}
Therefore we have
\begin{equation}
\frac{P_\mu(X_t > tc, \tau > t)}{P_\mu(\tau > t)} \sim \frac{(\alpha + c)^{-\kappa} G(t)}{\alpha^{-\kappa}G(t)}
\end{equation}
which gives us the desired result.
\end{proof}

\begin{ex}
\label{xmpl:half-cauchy}
If $\mu$ is a Half-Cauchy distribution (Cauchy distribution supported on $\rr^+$), the limiting distribution of $\displaystyle P_\mu\left(\left. \frac{X_t}{t}  > c \; \right| \; \tau > t\right)$ is Lomax distribution with shape parameter 1 and scale parameter $\alpha$.
\end{ex}

Note that when $\beta = 0$, $\mu$ is a distribution with regular or slowly varying tail. In such cases it is often more convenient to work with the asymptotic result $P_\mu(X_t > \overline{R}(t,c), \tau > t) \sim \mu([t\alpha + \overline{R}(t,c), \infty))$ directly to find the right scaling factor $\overline{R}$. We conclude this section with showing the quasi-limiting behavior of $\mu$ which  itself has slowly varying tail.

\begin{corollary}\label{thm:sv}
Suppose $\mu([x, \infty)) \sim  \frac{1}{\ln x}$ as $x\to\infty$. Then
$$\lim_{t \to \infty} P_\mu\left(\left. \frac{\ln X_t}{\ln t}  > c \; \right| \; \tau > t\right) = \begin{cases} 1 & c \le 1; \\ \frac{1}{c} & c > 1. \end{cases}$$
 that is, the limiting distribution is Pareto distribution with shape parameter $1$ and scale parameter $1$.
\end{corollary}

\begin{proof}
$\mu([x,\infty)) \sim \exp(-\ln \ln x)$ so we can apply Corollary \ref{rvprop}. Since  we have $\overline{R}(t,c) = t^c$, 
\begin{equation}
\begin{split}
\frac{P_\mu(X_t > t^c, \tau > t)}{P_\mu(\tau > t)} &\sim \frac{\ln (t\alpha)}{\ln (t\alpha + t^c)} \\
&\sim \begin{cases} \frac{\ln t + \ln \alpha}{\ln t + \ln \alpha} \to 1 & c < 1 \\ \frac{\ln t + \ln \alpha}{\ln t + \ln (\alpha + 1)} \to 1 & c = 1 \\ \frac{\ln t + \ln \alpha}{ c \ln t} \to \frac{1}{c} & c > 1\end{cases}
\end{split}
\end{equation}
which gives us the desired result.
\end{proof}

Notice that in our last example with super-heavy tail initial distribution, the scaled limiting distribution does not depend on the drift parameter $\alpha$ of the BM.

%% file: appendix.tex
\appendix
\section{Appendix} 
\subsection{Proof of Proposition \ref{prop:QLD}}\label{appendix1}
\begin{proof}
Suppose $\pi$ is a QLD for $\mu$. Then for every  and continuous function $f$ we have 
$ E_{\mu} [ f (X_t) | \tau>t] \to \int f d\pi$. That is, $E_{\mu} [f(X_t),\tau_t] \sim P_{\mu}(\tau>t) \int f d\pi$, provided $\int f d\pi\ne 0$. Fix such $f$ and let $t_1,t_2>0$. Then by the Markov property, 
\begin{equation}\label{eq:move_time}  E_{\mu} [ f (X_{t_1+t_2}) , \tau>t_1+t_2 ] = E_{\mu} [ h_{t_2}(X_{t_1}),\tau>t_1].
\end{equation} 
where $h_{t_2}(x) = E_x [f(X_{t_2}),\tau>t_2]$. By our assumption $h_{t_2}(\cdot)$ is continuous and bounded,  and therefore, $E_{\mu} [h_{t_2}(X_{t_1})|\tau>t_2]\to \int h_{t_2} d \pi$. On rewriting \eqref{eq:move_time} we have 
$$ E_{\mu}[ f (X_{t_1+t_2})| \tau >t_1+t_2] = E_{\mu} [h_{t_2}(X_{t_1})|\tau>t_1] \times \frac{P_{\mu}(\tau>t_1)}{P_{\mu}(\tau>t_1+t_2)}.$$ 
By our assumption, as $t_1\to\infty$  the lefthand side converges to the positive limit $\int f d\pi$ and the first expression on the righthand side converges to $\int h_{t_2}d \pi$. Therefore the ratio on the righthand side converges to a nonzero limit we denote by $R(t_2)$. This limit is independent of the choice of $f$. Therefore 
\begin{equation} 
\label{eq:QLD2QSD} \int f \pi = E_{\pi}  [ f(X_{t_2}),\tau>t_2] R(t_2)=E_{\pi} [f (X_{t_2}) | \tau>t_2] R(t_2) P_{\pi}(\tau> t_2).
\end{equation} 
Taking $f\equiv1$, we obtain $R(t_2) = \frac{1}{P_{\pi}(\tau>t_2)}$, and plugging this in back into \eqref{eq:QLD2QSD} gives proves the claim. 
\end{proof} 

\subsection{Proof of Proposition \ref{prop:heavytail}} \label{appendix2}
\begin{proof}
By the Markov property,
\begin{equation}
\begin{split}
P_\mu(\tau > s+t) &= P_\mu(\tau > t, P_{X_t} (\tau > s)) \\
&= P_\mu(P_{X_t} (\tau > s) \; | \; \tau > t) P_\mu (\tau > t)
\end{split}
\end{equation}
Write $f(x) = P_x (\tau > s)$. 
$\pi$ is a QSD and therefore the distribution of $\tau$ under $P_{\pi}$ is exponential with a parameter $\lambda_\pi>0$. Since $\pi$ is the QLD of $\mu$, for arbitrary $\epsilon > 0$ there is some $t_0=(t_0,\mu,s)$ such that for each $t > t_0$,
\begin{equation}
\Bigl|P_\mu(P_{X_t} (\tau > s) \; | \; \tau > t) - \mathbb{E}_\pi (f)\Bigr| < \epsilon e^{-\lambda_\pi s}
\end{equation}
Since $\mathbb{E}_\pi(f) = P_\pi(\tau > s) =e^{-\lambda_\pi s}$, we have that 
\begin{equation}
\label{eq:closetolambda} P_\mu(P_{X_t} (\tau > s) \; | \; \tau > t) \le (1+ \epsilon) e^{-\lambda_\pi s},~t>t_0.
\end{equation} 
Choose $s=1$, and apply \eqref{eq:closetolambda} repeatedly to obtain 
\begin{equation}
\begin{split}
P_\mu(\tau > t_0+1) &\le (1+\epsilon) e^{-\lambda_\pi } P_\mu(\tau > t_0) \\
P_\mu(\tau > t_0+2) &\le (1+\epsilon) e^{-\lambda_\pi}  P_\mu(\tau > t_0+1) \le (1+\epsilon)^2 e^{-2\lambda_\pi} P_\mu(\tau>t_0) \\
& \vdots \\
P_\mu(\tau > t_0 + n) &\le (1+\epsilon)^n e^{-n\lambda_\pi} P_\mu(\tau> t_0)
\end{split}
\end{equation}
Since the choice of $\epsilon$ is arbitrary the result follows. 
\end{proof}

\subsection{QSDs for BM with constant drift} \label{appendix3}

Here we provide a formal derivation for densities of the QSDs under assymption \ref{assume:drifted_BM}. Recall that a BM with constant drift $-\alpha$ on  $\rr_+$  absorbed at $0$ is the sub-Markovian process generated by  ${\cal L}_{\alpha}$, which for each $u$ satisfying $u \in C^2(\rr_+)$ and $u(0) = 0$, 
$$ {\cal L}_\alpha u = \frac 12 u'' - \alpha u'.$$ 

The formal adjoint ${\cal L}_\alpha^*$ of ${\cal L}_\alpha$, with respect to integration by parts, is given by

$${\cal L}^* v = \frac{1}{2} v'' + \alpha v' ,~ v \in C^2(\rr_+), v(0)=0.$$

Observe that for any $f$ in the domain of ${\cal L}_\alpha$,
\begin{equation}
\begin{split}
\frac{d}{dt}P_x(f(X_t), \tau > t) &= {\cal L}_\alpha P_x(f(X_t), \tau>t)\\
\Rightarrow P_x(f(X_t), \tau > t) &= f(x) + \int_0^t {\cal L}_\alpha P_x(f(X_s), \tau > s) ds
\end{split}
\end{equation}

Suppose a probability density function $\pi$ satisfies ${\cal L}_\alpha^*\pi = -\lambda \pi$ for some $\lambda > 0$. Notice that every QSD must be smooth, since if $\pi$ is a QSD then by definition we have the following density.
\begin{equation}
\begin{split}
    \pi(y) &= P_\pi(X_s = y \; | \; \tau > s) \\
    &= \frac{P_\pi (X_s = y, \tau > s)}{P_\pi (\tau > s)}
\end{split}
\end{equation}

Then with integration by parts we have the following.
\begin{equation}
\label{operator1}
\begin{split}
E_\pi (f(X_t), \tau > t) &= \int f(x) \pi(x) dx + \int \int_0^t {\cal L}_\alpha \left(E_x(f(X_s), \tau_s)\right) ds \pi(x) dx \\
&= \int f(x) \pi(x) dx + \int_0^t \int E_x(f(X_s), \tau > s) {\cal L}_\alpha^*\pi(x) dx ds\\
&= \int f(x) \pi(x) dx - \lambda \int_0^t E_\pi (f(X_s), \tau > s) ds
\end{split}
\end{equation}

Setting $\displaystyle h(t) = E_\pi (f(X_t), \tau > t)$, \eqref{operator1} gives
\begin{equation}
h'(t) = -\lambda h(t) \quad \Rightarrow \quad E_\pi(f(X_t), \tau > t) = e^{-\lambda t} \int f(x) \pi(x) dx
\end{equation}

Therefore by monotone convergence,

$$P_\pi (\tau > t) = e^{-\lambda t}$$
$$E_\pi(f(X_t) | \tau > t) = \int f(x) \pi(x) dx$$

That is, $\pi$ is a density of a QSD if and only if ${\cal L}_\alpha^*\pi = -\lambda \pi$. We can see that a density for a QSD $\pi$ is a solution to standard ODE and depends on the parameter $\lambda$. The range of $\lambda$ for which such a density exists is  $\lambda \in (0,\alpha^2/2]$ and for each such $\lambda$ corresponds a unique density. Fix such $\lambda$, set $\gamma = \sqrt{\alpha^2 - 2\lambda}$, and let $\pi_\gamma$ be the corresponding density. Then 
\begin{equation} 
\label{eq:pigamma}
\pi _\gamma(x) =\begin{cases} \frac{\alpha^2-\gamma^2}{\gamma } e^{-\alpha  x} \sinh(\gamma x) & \gamma >0\\
 \alpha^2 x e^{-\alpha x} & \gamma =0.\end{cases}
 \end{equation} 